\documentclass[10	pt, oneside,reqno]{amsart}
\usepackage{amsfonts,amsthm,mathrsfs,amssymb}
\usepackage{amsmath}
\usepackage{verbatim}
\usepackage[active]{srcltx}
\usepackage[matrix,arrow,curve]{xy}
\usepackage{verbatim}
\usepackage{bbm}
\usepackage{setspace}
\usepackage[margin=2.5cm]{geometry}

\newtheorem{thm}{Theorem}[section]
\newtheorem{prop}[thm]{Proposition}
\newtheorem{lemma}[thm]{Lemma}
\newtheorem{cor}[thm]{Corollary}

\newtheorem{thm*}{Theorem}

\numberwithin{equation}{section}

\newcommand{\reg}{\mathrm{reg}}

\theoremstyle{definition}

\makeatletter
\def\fourvdots{\vbox{\baselineskip1\p@ \lineskiplimit\z@
  \kern6\p@\hbox{.}\hbox{.}\hbox{.}\hbox{.}}}
\makeatother

\newcommand{\an}{\mathrm{an}}

	\usepackage{verbatim}

\newcommand{\scrI}{\mathscr{I}}

\newcommand{\scrS}{\mathscr{S}}

\newcommand{\calF}{\mathcal{F}}

\newcommand{\calO}{\mathcal{O}}

\newcommand{\calS}{\mathcal{S}}

\newcommand{\fraka}{\mathfrak{a}}

\newcommand{\frakd}{\mathfrak{d}}

\newcommand{\frakg}{\mathfrak{g}}

\newcommand{\frakk}{\mathfrak{k}}
\newcommand{\frakl}{\mathfrak{l}}

\newcommand{\frakp}{\mathfrak{p}}

\newcommand{\frakH}{\mathfrak{H}}

\newcommand{\Ads}{\mathbb{A}}

\newcommand{\CC}{\mathbb{C}}

\newcommand{\QQ}{\mathbb{Q}}
\newcommand{\RR}{\mathbb{R}}

\newcommand{\WW}{\mathbb{W}}

\newcommand{\ZZ}{\mathbb{Z}}

\renewcommand{\ker}{\operatorname{ker}}

\newcommand{\isomto}{\xrightarrow{\raisebox{-5pt}{$\sim$}}}

\newcommand{\rar}{\rightarrow}

\newcommand{\hra}{\hookrightarrow}	
\newcommand{\thrar}{\twoheadrightarrow}

\newcommand{\wt}{\widetilde}
\newcommand{\wh}{\widehat}

\newcommand{\ol}{\overline}

\newcommand{\End}{\mathrm{End}}
\newcommand{\Hom}{\mathrm{Hom}}

\newcommand{\Lie}{\mathrm{Lie}}

\newcommand{\Gal}{\mathrm{Gal}}

\newcommand{\Tr}{\operatorname{Tr}}
\newcommand{\tr}{\operatorname{tr}}
\newcommand{\Res}{\operatorname{Res}}

\newcommand{\Sp}{\operatorname{Sp}}

\newcommand{\Ut}{\operatorname{U}}
\newcommand{\pr}{\mathrm{pr}}
\newcommand{\Span}{\mathrm{Span}}

\newcommand{\Ind}{\mathrm{Ind}}

\newcommand{\rank}{\mathrm{rank}}

\newcommand{\Herm}{\mathrm{Herm}}

\newcommand{\Nm}{\mathrm{Nm}}
\newcommand{\ab}{\mathrm{ab}}

\newcommand{\Alg}{\mathrm{Alg}}

\renewcommand{\Im}{\operatorname{Im}}
\renewcommand{\ker}{\operatorname{ker}}
\renewcommand{\Re}{\operatorname{Re}}

\newcommand{\codim}{\operatorname{codim}}

\newcommand{\GL}{{\mathrm{GL}}}

\newcommand{\SU}{{\mathrm{SU}}}

\renewcommand{\mod}{\text{mod }} 

\newcommand{\vol}{\mathrm{vol}}

\newcommand{\dash}{\textendash}

\newcommand{\KM}{\mathrm{KM}}

\title{Generating series of intersection volumes of special cycles on unitary Shimura varieties}
\author{Z. Amir-Khosravi}
\newcommand{\Ik}{\operatorname{Ik}}

\begin{document}
\maketitle
\begin{abstract}We form a generating series of regularized volumes of intersections of special cycles on a non-compact unitary Shimura variety with a fixed base change cycle. We show that it is a Hilbert modular form by identifying it with a theta integral, with parameters lying outside the classical convergence range of Weil, which we show converges. By applying the regularized Siegel-Weil formulas of Ichino and Gan-Qiu-Takeda, we show that modular form is the restriction of a hermitian modular form of degree $n$ related to Siegel Eisenstein series on $U(n,n)$. These results follow from a computation showing the Kudla-Millson Schwartz function vanishes under the Ikeda map. 
\end{abstract}

\tableofcontents
\section*{Introduction}

Let $E/F$ be a CM extension of number fields, with $n = [F:\QQ]$, and $(V,Q)$ a 
non-degenerate 
hermitian space over $E$ of dimension $m$. Put
\begin{align}H = \Res_{F/\QQ} \Ut(V,Q)
\end{align}
and assume $H(\RR)$ is non-compact. Denote by $D$ the set of maximal 
negative-definite subspaces of 
$V(\RR)$. Then $D$ is a hermitian symmetric domain on which $H(\RR)$ acts 
transitively.

Let $L\subset V$ be a maximal lattice (in the sense of \cite{Shim64}), and 
let $\Gamma(L)$ denote its stabilizer in $H(\QQ)$. If $\Gamma \subset 
\Gamma(L)$ is a suitable arithmetic subgroup, the quotient 
\begin{align}M = \Gamma\backslash D\end{align}
is a complex algebraic variety.

The variety $M$ hosts a collection of \textit{special cycles} 
$C_\beta$, indexed by $\beta \in \Herm_r(E)$, the set of $E$-hermitian $r\times 
r$ matrices, for $1 \leq r \leq m$. We encode the cycles $C_\beta$ in a formal generating series as follows. 

Recall the complex Siegel half-space of degree $r$
$$\CC\frakH_r = \{ \tau = u + iv: ,u,v\in \Herm_r(\CC),\ \ v > 0\}.$$
For each $T=(\tau_1,\cdots, \tau_n)\in \CC\frakH_r^n$, $\beta=(\beta_1,\cdots, \beta_n)\in 
\Herm_r(E)^n$, write
$$ e_*(\beta T) = \exp(2\pi i \sum_i \Tr \beta_i \tau_i).$$
A remarkable fact is that
$$ \calF(T) = \sum_{\beta\in \Herm_r(E)} C_\beta \cdot e_*(\beta T)$$
behaves like the Fourier expansion of a Siegel automorphic form on $U(r,r)^n$. 
More concretely, let $C\subset M$ 
be an algebraic cycle of dimension $r$. One may consider the intersection 
numbers 
$I(C,C_\beta)$ of $C$ and $C_\beta$, in some suitable sense. Then 
the series
$$ \calF(T;C) = \sum_{\beta\in \Herm_r(E)} I(C,C_\beta) e_*(\beta T)$$ 
should, under ideal conditions, be the Fourier expansion of an actual automorphic form. 

The work of Hirzebruch and Zagier on intersection numbers of curves on Hilbert modular varieties, and that of Kudla  
\cite{KBalls78} on Picard modular varieties, are the first results of this type. The phenomenon was studied in 
great generality by Kudla and Millson in the 1980s \cite{KM86, 
KM88, KM90IHES}, with $\calF(T;C)$ understood as a kind of geometric theta 
lift of the cycle $C$. If $V$ is anisotropic, so that $M$ is compact, their work proves all the expected properties of $\calF(T)$. One may also replace $(V,Q)$ by a quadratic space over $F$, and obtain cycles $C_\beta$ on arithmetic quotients of symmetric space that may not be algebraic. The geometric theta 
lift then produces automorphic forms on $\Sp(2nr)$. 

More recently, Funke and Millson have extended and generalized some of these results to the non-compact case \cite{FunkMill02, 
FunkMill06, FunkMill11, FunkMill13}. Although they mainly focus on the 
orthogonal case, their methods are expected to also apply \textit{mutatis mutandis} to the unitary case.

In this paper, we focus on a particular unitary case where $M$ is a non-compact, and $C$ is a \textit{base change} cycle, obtained from a model $(V_0,Q_0)$ of $(V,Q)$ 
over an imaginary quadratic field $E_0/\QQ$, so that
$$(V,Q)=(V_0,Q_0)\otimes_{E_0} E.$$

We assume that $(V_0,Q_0)$ has signature $(n,1)$, and 
take $r=1$. We put
\begin{align}H_0 = \Ut(V_0,Q_0).
\end{align}
If $D_0$ denotes the hermitian symmetric domain associated to $H_0$, $D$ may be 
identified with $D_0^n$. The injection $H_0 \hra H$ then induces a diagonal 
embedding $D_0\hra D_0^n$.  For a suitable arithmetic subgroup $\Gamma_0 
\subset H_0(\QQ)\cap \Gamma$, we put
\begin{align}M_0 = \Gamma_0 \backslash D_0,
\end{align}
and let $C_0$ denote the cycle given by the induced locally finite map 
\begin{align}\iota_0: M_0 \rar M.
\end{align}

We will have reason to consider different lattices $L\subset V$, so we write 
$C_{\beta,L}$ to emphasize the dependence. For non-zero $\beta\in F = 
\Herm_1(E)$, the cycles $C_{\beta,L}$ and $C_0$ have complementary codimension. 
By $I(M_0,C_{\beta,L})$ we denote the intersection number of $C_{\beta,L}$ and 
$M_0$, in a suitable sense. 

Let $K(L)\subset H_0(\Ads_f)$ denote the stabilizer of $L\otimes \wh{\ZZ} 
\subset V(\Ads_f)$, and fix a set of representatives $h_1,\cdots, h_l$ of
$$ H_0(\QQ)\backslash H_0(\Ads_f)/K_{H_0}.$$
Denote by $L_i \subset V_0^n\simeq V$ the lattices 
\begin{align}L_i = 
h_i (L\otimes \wh{\ZZ})\cap V_0(\QQ)^n,\ \ \ i=1,\cdots, l.
\end{align}

For $\tau=(\tau_1,\cdots, \tau_n)\in \frakH^n$, we write $\tau_i = u_i + i 
v_i$, with $u_i, v_i\in \Herm_n(\CC)$, and put
\begin{align}F(\tau) = c_0 + i^{-n} \prod_{i=1}^n |v_i|^{(n+1)/4} 
\sum_{\beta\in F}\left( \sum_{i=1}^r I(C_0,C_{\beta,L_i})\right)e_*(\beta 
\tau)  
,\end{align}
where $c_0$ is some explicit constant. 

The aim of this paper is to show 
$F(\tau)$ is the Fourier expansion 
of a Hilbert modular form of weight $\frac{n+1}{2}$ on $\frakH^n$, and to 
identify it with the restriction to $U(1,1)^n$ of an automorphic form on 
$U(n,n)$ derived from Siegel Eisenstein series.

We now describe the results in more detail.

Let $(W,J)$ be a split skew-hermitian space of signature $(1,1)$, and put
$$(W_0,J_0)= (\Res_{E/E_0} W, \tr_{E/E_0} J).$$
Define
$$G_0 = U(W_0,J_0),\ \ \ G = \Res_{F/\QQ}U(W,J),$$
so that $G_{0,\RR}= U(n,n)$ and $G_\RR\simeq U(1,1)^n$. One has a natural 
inclusion $G\subset G_0$ that ``forgets'' $F$-linearity.
 
After fixing a non-trivial additive character $\psi$ on $\Ads_\QQ/\QQ$ and a 
suitable Hecke character $\chi$ on $\Ads_{E_0}^\times/E_0$, one obtains a Weil 
representation
$$ \omega : G_0(\Ads)\times H_0(\Ads) \rar 
\End(\calS)$$
where $\calS = \scrS(V_0(\Ads)^n)$ is a space of adelic Schwartz functions.
To each $\phi\in \calS$ one may associate the theta kernel
$$ \Theta(g,h;\phi) = \sum_{x\in V_0(\Ads)^n} \omega(g,h)\varphi(x).$$
Here $\omega(g,h)\varphi(x)$ denotes $\omega(g,h)\varphi$ 
evaluated at $x$. 
There is an associated theta 
integral
\begin{align}\label{TI} I(g,\phi) = \int_{H_0(\QQ)\backslash H_0(\Ads)} 
\Theta(g,h;\phi) 
dh.\end{align}
The Witt index $r$ of $(V_0,Q_0)$ is equal to $1$. Since $V_0$ is not 
anisotropic and $m-r \leq n$, by a theorem of Weil \cite{Weil65Acta} the above 
integral will in general diverge. For a specific $\varphi\in S$ derived from 
a construction of Kudla and Millson, we will 
prove:

\bigskip
\textbf{Theorem A.}{\ \it For $g\in G(\Ads)\subset G_0(\Ads)$, the integral 
$I(g,\varphi)$ is absolutely convergent and defines an automorphic form on 
$G(\Ads)$. Its restriction to $G(\RR) \subset G(\Ads)$ can be identified 
with 
$F(\tau)$.}
\bigskip

To identify $F(\tau)$ with the 
restriction of an automorphic form on $U(n,n)$, 
we consider regularized theta 
integrals of the form
$$ B(g,s;\phi) = \int_{H_0(\QQ)\backslash H_0(\Ads)} \Theta(g,h;\omega(z) \phi) 
E_{H_0}(h,s) dh, \ \ \ \Re(s) \gg 0.$$
Here $E_{H_0}(h,s)$ is a particular auxiliary Eisenstein series on $H_0$, and 
$\omega(z)$ is a regularizing Hecke operator acting on $\calS$. The integral 
converges absolutely for $\Re(s)$ large and has a meromorphic continuation to 
the $s$-plane. It has a pole of 
order at most two at $s=s_0=\frac{1}{2}$, and a Laurent expansion
$$ B(g,s;\phi) = B_{-2}(\phi) (s-s_0)^{-2} + B_{-1}(\phi) 
(s-s_0)^{-1} + B_0(\phi) + B_1(\phi) (s - s_0) + \cdots,$$ 
where each coefficient $B_k(\phi)$ is an automorphic form on $G_0(\Ads)$.

On the other hand, to $\phi\in \calS$ one associates in a standard way a 
\textit{Siegel-Weil section}  $\Phi^{(s)}\in \Ind_{P_0(\Ads)}^{G_0(\Ads)} 
\chi_0 |\det|^s_{\Ads_{E_0}}$, which is a family of functions on $G_0(\Ads)$ 
parametrized by $s\in \CC$. The corresponding Siegel Eisenstein series is 
$$ E(g,s;\phi)= \sum_{\gamma\in P_0(\QQ)\backslash G_0(\QQ)} 
\Phi^{(s)}(\gamma 
g),\ \ \ g\in G_0(\Ads).$$
The series has a meromorphic continuation to the $s$-plane, with at most a 
simple pole at $\rho = \frac{m-r}{2} = 1$. Its Laurent 
expansion at $s=\rho$ is written 
$$ E(g,s;\phi) = A_{-1}(\phi) (s-\rho)^{-1} + A_0(\phi) + A_{1}(\phi) (s-\rho) 
+ \cdots,$$
where each $A_{-1}(\phi)$ is again an automorphic form on $G_0(\Ads)$. 

To state the result, we must mention that to $\phi\in \calS$ one may associate 
another $\phi'\in \calS$ by a certain procedure. It involves lifting it to 
$\wt{\phi}$ in the Weil representation of a larger group, 
applying the derivative of a normalized intertwining operator to 
$\Phi^{(s)}(\wt{\phi})$, and restricting the corresponding Schwartz function 
back to $\calS$.  When  $\phi=\varphi$ as before, we prove:

\bigskip

\textbf{Theorem B.\ \ }{\it As automorphic forms on $G_0(\Ads)$, 
$$B_{-1}(\varphi)=  A_0(\varphi)  - \frac{1}{2}A_{-1}(\varphi').$$
	Furthemore, for $g\in G(\Ads)$ we have 
	$$B_{-1}(\varphi)(g) = I(g,\varphi).$$}

The proof is by first applying the regularized Siegel-Weil formulas of Ichino 
\cite{Ich04MZ} to show that $B_{-2}(\varphi)=0$. This allows us to simplify and 
refine the second-term identity of \cite{GQT} to obtain the precise expression 
for $B_{-1}(\varphi)$.

\section{Notation}

We fix an ordering $\lambda_1,\cdots, \lambda_g$ of the distinct embeddings 
$F\hra \RR$, and use both $z^\sigma$ and $\ol{z}$ to denote complex conjugation 
on $E$ 
and $E_0$.
If $U$ is a $k$-vector space and $A$ is a $k$-algebra, we write $U(A)$ for 
$U\otimes_k A$. The finite adeles are denoted $\Ads_f$. For a lattice $L\subset 
V$, we put $\wh{L} = L\otimes \wh{\ZZ} 
\subset V(\Ads_f)$, where $\wh{\ZZ} = \varprojlim \ZZ/n\ZZ$.

$(V,Q)$ will denote a non-degenerate hermitian space 
of dimension $m$ over $E$. We have
$$ V(\RR) \simeq V^{(1)}\oplus \cdots \oplus V^{(g)},$$
where each $V^{(i)} = V\otimes_{F,\lambda_i} \RR$ is a 
complex hermitian space of dimension $m$. 

We fix a maximal isotropic subspace  $X \subset V$, so that $r=\dim X$ is the 
Witt index of $(V,Q)$. We fix an ordered basis $(x_1,\cdots, x_r)$ for $X$. We 
fix a dual isotropic subspace $X^*\subset V$ spanned by vectors $(x_1^*,\cdots, 
x_r^*)$ 
satisfying $Q(x_i,x_j^*) = \delta_{ij}$. We thus occasionally identify $X^*$ 
with 
$\Hom(X,E)$.
There's an orthogonal decomposition
\begin{align}\label{Van} V = X \oplus V_\an \oplus X^*,\end{align}
where $V_\an$ is an anisotropic subspace of dimension $m_0 = m-2r$.

As in the introduction, we fix a maximal lattice $L\subset V$. By the 
classification of such lattices in \cite{Shim64}, we can assume $(x_1,\cdots, 
x_r)$ and $(x_1^*,\cdots, x_i^*)$ have been chosen in such a way that
$$ L = \fraka_1 x_1 \oplus  \dots \oplus  \fraka_r x_r \oplus L' 
\oplus \fraka_1^* x_1^*  \oplus \cdots \oplus \fraka_r^{*} x_r^*.$$
Here $\fraka_i\subset E$ are fractional ideals, $\fraka_i^* = 
\ol{\fraka_i}^{-1} \delta_E^{-1}$, and $L' \subset V_\an$ is a maximal lattice.

We let $P_{H}\subset H$ denote the stabilizer of $X$, and put $K_H = 
K_{H,\infty}K(L)$, where $K(L)\subset H(\Ads_f)$ is the stabilizer of 
$\wh{L}\subset V(\Ads_f)$, and $K_{H_\infty} = H(\RR)\cap U(m)$. 

For $n>0$, we let $(W_n,J_n)$ denote the split skew-hermitian space of 
dimension $2n$ spanned by a basis $e_1, \cdots, e_n , f_1,\cdots,f_n$ with 
respect to which $J_n$ is given by
\begin{align}\label{Jn}J_n(x,y) = {}^t \ol{y}\left(\begin{array}{cc}& I_n\\-I_n 
& 
\end{array}\right) x,\ \ \ x,y\in E^{2n}.\end{align}
We put  \begin{align}G_n=\Res^F_\QQ U(W_n,J_n),\end{align}
and $G=G_n$ when $n$ is fixed.

Let $Y = \Span\{e_1,\cdots, e_n\}$, $Y^* = \Span\{f_1,\cdots, f_n\}$, so that 
$Y$ and $Y^*$ are dual isotropic subspaces, and $W = Y \oplus Y^*$. 
We will sometimes identify $Y^*$ with $\Hom(Y,E)$.

Let $P$ denote the Siegel 
parabolic stabilizing $Y$, $M$ its Levi factor, and $N$ the unipotent 
radical. We may identify these with
$$M= \GL(Y),\ \ \ N = \Hom_F(Y,Y^*),\ \ \ P=NM.$$
More explicitly, using the fixed basis for $W$ we can write
\begin{flalign*}
 P(\Ads) &= \left\{p(a) = \left(\begin{array}{cc}a & * \\ & {}^t 
\ol{a}^{-1}\end{array}\right):\  a\in \GL_n(\Ads_E)\right\},\\
 M(\Ads) &= \left\{ m(a)=\left(\begin{array}{cc} a &  \\ & {}^t \ol{a}^{-1} 
\end{array}\right): a\in \GL_n(\Ads_E) \right\},\\
N(\Ads) &= 
\left\{n(b)=\left( 
\begin{array}{cc} 1 & b \\ & 1 \end{array}\right): b\in 
\Herm_n(\Ads_E)\right\}.
\end{flalign*}
where $\Herm_n(k)$ denotes $n\times n$ hermitian matrices with entries in $k$, 
for any $E$-algebra $k$. 

Let $K_G  = G(\wh{\ZZ}) \subset G(\Ads)$ be the 
standard maximal compact of 
$G(\Ads)$. 

For $g\in G(\Ads)$, we have a decomposition \begin{align} g = n(b) 
\cdot m(a) \cdot k,\ \ \ a \in \GL_n(\Ads_E), \ \ \ b\in 
\Herm_n(\Ads_E),\  \ \ k\in K_G.
\end{align}
We write $a(g) = a$ under this decomposition.

As in the introduction, we put
$$ W_0 = \Res_{E/E_0} W,\ \ \ J_0 = \tr_{E/E_0} J.$$
We explicitly identify $(W_0,J_0)$ with $(W_{ng},J_{ng})$ over $E_0$ as 
follows.

Fix a $\ZZ$-basis 
$\{r_1,\cdots, r_g\}$ for $\calO_F$, and a trace-dual basis
$\{s_1,\cdots, s_g\}$ for $\frakd_F^{-1}$. We then have
corresponding isomorphisms
\begin{align}\label{bb} \beta: \ZZ^g \rar \frakd_{F}^{-1},\ \ \ (x_1,\cdots, 
x_g) \mapsto 
\sum_{i=1}^g s_i x_i,\ \ \ \ \ \beta': \ZZ^g \rar \calO_F,\ \ \ (x_1,\cdots, 
x_g)\mapsto \sum_{i=1}^g r_i x_i.\end{align}
The $E_0$-linear map 
\begin{align}\label{RW}\beta^{\prime n}_{E_0} \oplus \beta_{E_0}^n: E_0^{ng} 
\oplus 
E_0^{ng} \rar
E^n \oplus E^n
\end{align}
then identifies $(W_{ng},J_{ng})$ over $E_0$ with $(W_0,J_0)$, using the fixed 
basis $\{e_i,f_i\}_{i=1}^n$ for $W$. 
Then For
\begin{align}G_0 = U(W_0,J_0),\end{align}
as in the introduction, $G_0(\RR)$ may be identified with the complex unitary 
gropu $U(ng,ng)$. 

The canonical inclusion $G\subset G_0$ that ``forgets'' $F$-linearity will be 
sometimes denoted
\begin{align} \jmath: G \hra G_0.\end{align}

\section{Cycles on unitary Shimura varieties}

To begin with, let $(V,Q)$ denote an arbitrary non-degenerate hermitian space 
over $E$ of 
dimension $m$. Assume that each $V^{(i)} = V\otimes_{F,\lambda_i} \RR$ has 
signature $(p_i,q_i)$, and that $p_i q_i \neq 0$ for some $i$. The 
hermitian symmetric domain $D$ 
associated to $H(\RR)$ has complex dimension $\sum_{i=1}^n p_i q_i$, and may be 
identified with
$$ D = \{ Z \subset V(\RR):\ \dim Z = \sum_{i=1}^n q_i,\ \ Q|_{Z} < 0\}.$$
For $Z_0\in D$ fixed, let $K_\infty \subset H(\RR)$ denote its stabilizer, so 
that $D\simeq H(\RR)/K_\infty$. 

An element $Z\in D$ consists of a maximal negative-definite subspace of 
$V\otimes \RR \simeq V^{(1)} + \cdots + V^{(g)}$, where the $V^{(i)}$ are 
mutually orthogonal complex hermitian spaces. Let $D^{(i)}$ denote the set of 
maximal negative-definite subspaces in $V^{(i)}$ (consisting of a single point if $q_i=0$). Since each $Z\in D$ can be written uniquely as $Z=Z^{(1)}+\cdots+Z^{(g)}$, with $Z^{(i)}\in D^{(i)}$, we can identify $D$ with $D^{(1)}\times \cdots \times D^{(g)}$.

Let $\Gamma(L) =  K(L)\cap H(\QQ)$, and assume $\Gamma\subset \Gamma(L)$ is an 
arithmetic subgroup 
 such that 
$\Gamma/Z(\Gamma)$ acts freely on $D$. Then as is well-known
\begin{align}  M  = D/\Gamma\end{align}
is a complex algebraic variety that admits a model over a number field.

We will now describe the family of \textit{special cycles} on $M$.

\subsection{Special cycles}

Let $x=(x_1,\cdots, 
x_g)\subset L $ be a collection of lattice vectors spanning a non-degenerate 
subspace $V_x'$ of $V$. Let $V_x = V_{x}'^\perp$ be the orthogonal complement, 
and write $Q_x = Q|_{V_x}$, $Q_x' = Q|_{V_x'}$. 

Put
\begin{align} H_x=\Res_{F/\QQ} \left(\Ut(V_x,Q_x) \times 
\Ut(V_x',Q'_x)\right).\end{align}
Then $H_x$ injects into $H$ canonically, and its symmetric space  may 
then 
be 
identified with
$$ D_x = \{Z\in D: Z = Z\cap V_x(\RR) + Z\cap V_{x}'(\RR)\}.$$

If each $V_X^{(i)}$ has signature $(r_i,s_i)$, then 
\begin{align}
\dim_\CC D_x = \sum_{i=1}^n (r_i s_i +(p_i-r_i)(q_i-s_i)).
\end{align}

Let $\Gamma_x =  H_x(\QQ)\cap \Gamma$, and put $M_x = D_x/\Gamma_x$. The \textbf{basic cycle} $C_x$ associated 
to $x$ is the image of the locally-finite map
$$ \iota_x: M_x  \rar M.$$

Now let $\beta\in \Herm_n(E)$ be a hermitian $g\times  g$ matrix, and put
\begin{align}\scrI_\beta = \{x = (x_1,\cdots, x_g)\in V^g: Q(x,x)=\beta\}.
\end{align}
Here $Q(x,x)$ denotes the matrix $(Q(x_i,x_j))$. The group $\Gamma$ acts on 
$V^g$ diagonally, and preserves $\scrI_\beta$. Let $y_1,\cdots, y_l$ be a 
set of representatives for the (finite number of) 
orbits of the action of $\Gamma$ on $\scrI_\beta$. 

The \textbf{special cycle} 
associated to $\beta\in \Herm_g(E)$ is defined to be
\begin{align}C_\beta = \coprod\limits_{i=1}^l C_{y_i}.
\end{align}

We remark that $C_\beta$ is a disjoint union of locally finite maps to $M$. 

Now suppose $(V,Q) = (V_0,Q_0)\otimes_{E_0} E$, where $(V_0,Q_0)$ is 
hermitian over $E_0$ of signature $(p_0,q_0)$, with $pq\neq 0$. For convenience we assume $V_0\subset V$. Composing the inclusion $V_0(\RR)\hra V(\RR)\simeq \prod_i V^{(i)}$ with each 
projection $\prod_i V^{(i)} \rar V^{(i)}$ we obtain isomorphisms $V_0(\RR) 
\isomto V^{(i)}$ of complex hermitian spaces. Let $H_0=U(V_0,Q_0)$ be as in the introduction, and identify its symmetric space $D_0$ with the space of negative-definite $q_0$-dimensional subspaces of $V_0(\RR)$. Then via the isomorphisms $V_0(\RR) \simeq V^{(i)}$, each $D^{(i)}$ is identified with $D_0$, and the map $D_0 \hra D$ induced by $H_0(\RR)$ with the diagonal $D_0 \hra D_0^n$.

Take 
$\Gamma_0\subset H_0(\QQ)\cap \Gamma$ to be an 
arithmetic subgroup such 
that $\Gamma_0/Z(\Gamma_0)$ acts freely on the hermitian symmetric domain $D_0$ 
of $H_0$. Put
\begin{align}M_0 = D_0/\Gamma_0.
\end{align}
We assume furthermore that $M_0$ has finite volume.

The \textbf{base change cycle} associated to $(V_0,Q_0)$ and $\Gamma_0$ is the image of locally finite map 
\begin{align}\iota_0: M_0 \rar M.\end{align}

\subsection{Intersection volumes}

Since $(V_0,Q_0)$ has signature $(p,q)$, the hermitian symmetric domain $D_0$ has complex dimension $pq$, and $D$ has dimension $npq$. 

For $x=(x_1,\cdots, x_r) \in V^r$, let  $(r_i,s_i)$ denote the signature of 
$V_{x}^{(i)}$. Then $D_x$ has 
dimension 
$$ \sum_{i=1}^n (r_i s_i + (p-r_i)(q-s_i)) = npq -\sum_{i=1}^n (ps_i + q r_i - 
2r_i s_i).$$
Then we have
$$ \codim C_0 = (g-1)pq,\ \ \ \codim C_x = \sum_{i=1}^n (p s_i + q r_i - 2 r_i 
s_i)$$
and $C_0$, $C_x$ have complementary codimension if and only if
\begin{align}\label{codimeq} \sum_{i=1}^n s_i(p-r_i) + r_i(q-s_i)=pq.\end{align}

From now on we assume $(p,q)=(n,1)$. 

\begin{lemma}Then $C_0$ and $C_x$ have complementary 
codimension if and only if either $(r_i,s_i)=(1,0)$ for all $i$, or 
$(r_i,s_i)=(p-1,1)$ for all $i$.
\end{lemma}

\begin{proof}
Each term of the sum in (\ref{codimeq}) is non-negative, there are $n$ terms, 
and the RHS is $n$. Hence the codimensions are complementary if and only if 
every summand is exactly 1. Now $0\leq s_i \leq q = 1$. If 
$s_i=0$, we must have $r_i=1$, and if $s_i=1$, $r_i = p-1$. Since $r_i+ s_i = 
\dim V_x$ for all $i$ and 
$n> 1$, the two cases are 
mutually 
exclusive.
\end{proof}

In fact the two cases in the lemma describe the same basic cycle 
$C_x$ coming from a decompositions $V=V_x + V_x'$, with $V_x$ and $V_x'$ switched. This is the reason for restriction to cycles with $r=1$.

We therefore consider the intersection of $C_0$ with $C_b$, for $b\in F = \Herm_1(E)$. By the lemma, the cycles have complementary codimension if and only if $b$ is totally positive. However, even in this case the intersection will not in general be transversal. Nevertheless it is a union of basic cycles on $M_0$, of varying codimensions, which all have finite volume since $M_0$ does, by assumption. 

Let $\xi\in V$. For $\lambda: F \hra \RR$, let $\lambda(\xi)$ denote the image of $\xi$ in $V_\lambda = V \otimes_{F,\lambda} \RR$, and $\xi^{\lambda}$ the pre-image of $\lambda(\xi)$ under $V_0(\RR)\simeq V_\lambda$. Write $V(\xi)$ for the subspace of $V_0(\RR)$ spanned by $\xi^{\lambda}$ for all $\lambda$, and $D_0(\xi)$ for $D_0\cap D_\xi$.

\begin{prop}If $V(\xi)$ is a positive-definite subspace of $V_0(\RR)$, then $D_0(\xi)$ is isomorphic to a complex ball of dimension $n-\dim V(\xi)$. Otherwise it contains at most one point.
\end{prop}
\begin{proof}Suppose $z_0\in D_0 \cap D(\xi)$. Then for each $\lambda$, the image of $z_0$ in $V_\lambda$ must be either orthogonal or proportional to $\xi^\lambda$. If $\lambda(Q(\xi,\xi))<0$, $z_0$ must be the line spanned by $\xi^{\lambda}$, hence $D_0(\xi)$ is either empty or only contains only $z_0$. If $\lambda(Q(\xi,\xi))>0$ for all $\lambda$, $z_0$ must be orthogonal to all $\xi^{\lambda}$, i.e. $V(\xi)\subset z_0^\perp$, which implies $V(\xi)$ is positive-definite since $V_0(\RR)$ has signature $(n,1)$. If $V(\xi)$ is positive-definite, then $D_0(\xi)$ is the set of negative-definite lines in $V(\xi)^\perp$, which has signature $(n-\dim V(\xi),1)$.
\end{proof}

Suppose that $x=(x_1,\cdots, x_n)\in V_0^n$ corresponds to $\xi\in V$ under our fixed isomorphism. In other words, $\xi = \sum_{j=1}^n r_j x_j$, where $(r_j)_j$ is the fixed $\ZZ$-basis for $\frakd_F^{-1}$. 
Let $V_{0,x}$ denote the subspace of $V_0$ spanned by $x_i$. Then $V_{0,x}(\RR) \supset V(\xi)$, since each $\xi^{(i)}=\sum_{j=1}^n \lambda_i(r_j) x_j\in V_{0,x}(\RR)$. On the other hand, the matrix $U=(u_{ij})\in \Herm_n(\CC)$ with $u_{ij}=\lambda_i(r_j)$ is invertible, so that $V_{0,x}(\RR)=V(\xi)$. 

If $Q(\xi,\xi)$ is positive-definite, the subspace $D_{0,x}\subset D_0$ associated to $x\in V_0^n$ coincides with $D(\xi)$ defined above. 

Let $K(\xi)$ denote the fiber product of the morphisms $\iota_{\xi}:D_{x}/\Gamma_{\xi} \rar D/\Gamma$ and $\iota_0: D_0/\Gamma_0 \rar D/\Gamma$ over $D/\Gamma$. In other words
$$ K(\xi) = \{(\Gamma_0 z, \Gamma_\xi w): z\in D_0,\ w\in D_\xi,\ \Gamma z = \Gamma w\}.$$

\begin{prop}\label{decbasicint} The map
	$$ \coprod_{[\gamma]\in \Gamma_0\backslash \Gamma/\Gamma_\xi} D(\gamma \xi)/(\Gamma_0 \cap \Gamma_{\gamma\xi}) \rar K(\xi),$$
	induced by $(z,\gamma)\mapsto (\Gamma_0 z, \Gamma_\xi \gamma^{-1} z)$ is a bijection.
\end{prop}
\begin{proof}Let $(\Gamma_0 z, \Gamma_\xi w)\in K(\xi)$, so that $z = \gamma w$ for some $\gamma\in \Gamma$. Then $z\in D_0 \cap \gamma D_{\xi} = D_0 \cap D_{\gamma \xi} = D(\gamma\xi)$. Conversely, given any $z\in  D(\gamma \xi)$, we have $(\Gamma_0 z, \Gamma_\xi \gamma^{-1} z)\in K(\xi)$. This shows that the map
	$$ \coprod_{\gamma\in \Gamma} D(\gamma \xi) \rar K(\xi),\ \ \ (z,\gamma) \mapsto (\Gamma_0 z , \Gamma_\xi \gamma^{-1} z)$$
is surjective. Now suppose $(z_1,\gamma_1)$ and $(z_2,\gamma_2)$ map to the same point in $K(\xi)$, so that 
$$ (\Gamma_0 z_1, \Gamma_\xi \gamma_1^{-1} z_1)= (\Gamma_0 z_2, \Gamma_\xi \gamma_{2}^{-1} z_2).$$
Then for some $\gamma_0\in \Gamma_0$, $\gamma_\xi \in \Gamma_\xi$, we have $ z_2 = \gamma_0 z_1$ and $\gamma_1^{-1} z_1 = \gamma_\xi \gamma_2^{-1} z_2$, hence $\gamma_1 \gamma_\xi \gamma_2^{-1} \gamma_0 z_1 = z_1$. Then since $\Gamma$ acts without fixed points on $D$, we have $\gamma_1 \gamma_\xi \gamma_2^{-1} \gamma_0 = 1$, so that $\gamma_1 = \gamma_0^{-1} \gamma_2 \gamma_\xi^{-1}$ and $[\gamma_1]=[\gamma_2]$ as double cosets in $\Gamma_0\backslash \Gamma / \Gamma_\xi$. 

Conversely, if $\gamma_1 = \gamma_0 \gamma_2 \gamma_\xi$, then $D(\gamma_1 \xi) = D_0 \cap D_{\gamma_0 \gamma_2 \xi} = D_0 \cap \gamma_0 D_{\gamma_2 \xi} = \gamma_0 D(\gamma_2\xi)$, so that $D(\gamma_1 \xi)$ and $D(\gamma_2 \xi)$ have the same image in $D_0/\Gamma_0$. This shows the images of $D(\gamma \xi)$ in $K(\xi)$ are disjoint, as $\gamma$ ranges over double coset representatives of $\Gamma_0 \backslash \Gamma/\Gamma_\xi$. 

Now suppose $(z_1,\gamma)$ and $(z_2,\gamma)$ map to the same point in $K(\xi)$. Then $\gamma_0 z_1 = z_2$ for some $\gamma_0\in \Gamma_0$, and $\gamma_\xi \gamma^{-1} z_1 = \gamma^{-1} z_2$, so that $z_2 = \gamma \gamma_x \gamma^{-1} \gamma_0^{-1} z_2$, hence $\gamma_0 =\gamma \gamma_x \gamma^{-1} \in \gamma \Gamma_x \gamma^{-1} = \Gamma_{\gamma \xi}$. Then $\gamma_0\in \Gamma_0\cap \Gamma_{\gamma \xi}$, and conversely for any such $\gamma_0$, the points $(\gamma_0 z_1, \gamma)$ and $(z_1,\gamma)$ map to the same in $K(\xi)$. This shows that the surjective map 
	$$ \coprod_{[\gamma]\in \Gamma_0\backslash \Gamma/\Gamma_\xi} D(\gamma \xi)\rar K(\xi)$$
	induces the bijection in the statement.
\end{proof}

Let $\calO_i = \Gamma \xi_i$, $i=1,\cdots, t$ be the distinct orbits of $\Gamma$ acting on
$$I_b(L)=\{\xi\in V: Q(\xi,\xi)=b\},$$
and let $\iota_b: \coprod_{i} D_\xi/\Gamma_\xi \rar D/\Gamma$ denote the disjoint union of the locally finite maps $\iota_{\xi_i}: D_{\xi_i}/\Gamma_{\xi_i} \rar D/\Gamma$. The special cycle $C_b$ is then the image of $\iota_b$. By the proposition, the fiber product of $\iota_b$ with $\iota_0: D_0/\Gamma_0 \rar D/\Gamma$ is given by
$$ \coprod_{i=1}^t \coprod_{[\gamma]\in \Gamma_0 \backslash \Gamma / \Gamma_{\xi_i}} D(\gamma\xi)/(\Gamma_0\cap \Gamma_{\xi}) \rar D/\Gamma,\ \ \ (z,\gamma)_{t} \mapsto (\Gamma_0 z, \Gamma_{\xi_t} \gamma^{-1} z).$$

Denote by $L_0\subset V_0^n$ the lattice corresponding to $L\subset V$ under the isomorphism $V_0^n \simeq V$, $(x_1,\cdots, x_n)\mapsto \sum_{i=1}^n r_i x_i.$
Let $r=(r_1,\cdots, r_n)$ considered as a column vector, and $\beta\in \Herm_n(E_0)$ the matrix $Q_0(x_i,x_j)$. Then
$$ Q(\xi,\xi) = {}^t r \beta r.$$
Now let
$$I_{\beta}(L_0)=\{x\in L_0: Q_0(x_i,x_j)=\beta.\}$$

Then map $x \mapsto \xi$ then establishes a bijection  
$$ \bigcup\limits_{\tiny \begin{array}{c} \beta \in \Herm_n(E_0)\\{}^t r\beta r = b\end{array}}I_{0,\beta}(L_0) \rar I_b(L).$$
For each $\xi\in I_b$, the orbit $\Gamma \xi\subset I_b$ consists of $\Gamma/\Gamma_\xi$ elements. On the other hand, the action of $\Gamma_0$ on $I_{0,\beta}$ splits each such $\Gamma$-orbit into $\Gamma_0$-orbits indexed by
$$ \Gamma_0\backslash \Gamma / \Gamma_\xi.$$

For each $x\in L_0$, let $\Gamma_{x}=\{\gamma\in \Gamma: \gamma x = x\}.$ Then we have a map
$$ \iota_{0,x}: D_{0,x}/(\Gamma_{x}\cap \Gamma_0) \rar D_0/\Gamma_0,$$
whose image is a basic cycle on $D_0/\Gamma_0$, that we denote by $C_{0,x}$. In fact $\Gamma_x = \Gamma_{\xi}\cap \Gamma_0$ and $D_{0,x} = D(\xi)$, so that these are the same cycles occurring in Prop. \ref{decbasicint}.

\begin{thm}\label{C0Cb} For each $b\in F$,
	$$C_0 \cap C_b = \coprod_{\tiny \begin{array}{c}\beta\in \Herm_n(E_0)\\ {}^t r \beta r = b\end{array}} C_\beta.$$ 
\end{thm}
\begin{proof}The cycle $C_0\cap C_b$ is a union of $C_0 \cap C_\xi$ as $\xi$ runs through representatives of $\Gamma$-orbits in $I_b(L)$. By Proposition \ref{decbasicint}, each $C_0 \cap C_\xi$ is a disjoint union of images of $C_{0,\gamma x}$, with $[\gamma]\in \Gamma_0 \backslash \Gamma /\Gamma_\xi$. On the other hand, each such $C_{0,\gamma x}$ is the basic cycle in $D_0/\Gamma_0$ associated to the $\Gamma_0$-orbit of $\gamma x$. 
\end{proof}

We wish to define an appropriate notion of volume for basic and special cycles on $D_0/\Gamma_0$, and correspondingly define the intersection volume of $C_0$ and $C_b$ by
\begin{align}I(C_0,C_b) = \sum_{\tiny \begin{array}{c}\beta\in \Herm_n(E_0)\\ {}^t r \beta r = b\end{array}} \vol(C_\beta).
\end{align}

We do this by means of a particular Schwartz-Bruhat function $\varphi_\KM\in \scrS(V_0(\RR)^n)$ constructed by Kudla and Millson, detailed in the appendix.

For $x\in V_0^n$, Let $C_x=D(x)/\Gamma(x)$ be the associated basic cycle in $D_0/\Gamma_0$. Since $D(x)$ is totally geodesic in $D_0$, there is a geodesic fibration $D_0 \rar D(x)$, which induces a map$\pi_x: D_0/\Gamma(x) \rar D(x)/\Gamma(x)$ whose fibers are geodesics. The function $f_x(h) = \varphi_\KM(h^{-1}x)$, \textit{a priori} defined on $H_0(\RR)$, descends to a function $f_x(z)$ on $D_0/\Gamma(x)$. It is integrable over $D_0/\Gamma(x)$, and for each $z\in D(x)/\Gamma(x)$, the fiber integral
$$ \int_{\pi^{-1}(z)} f_x(z) dz,$$
is \textit{independent} of $z$, and equal to $i^{-n} \exp(-\pi \Tr Q_0(x,x))$. Now we define the (regularized) volume $\vol(C_x)$ by the relation
\begin{align}\int_{D_0/\Gamma(x)} f_x(z) dz = \vol(C_x) i^{-n} \exp(-\pi \Tr Q_0(x,x)).
\end{align}
Here $dz$ denotes the canonical volume element induced by the Bergman metric on $D(x)$. Note that if $C_x$ has finite volume, then 
$$ \int_{D_0/\Gamma(x)} f_x(z) dz = \int_{C_x}\int_{\pi^{-1}(w)} f_x(z) dz dw = i^{-n} \exp(-\pi \Tr Q(x,x)) \int_{C_x} dw = \vol(C_x) i^{-n} \exp(-\pi \Tr Q_0(x,x).$$

\section{Theta integrals}

\subsection{Weil representation}

We at first let $(V,Q)$ be a general non-degenerate hermitian space over $E$ of dimension $m$, and $(W,J)=(W_n,J_n)$ the split skew-hermitian space over $E$ of dimension $2n$. Let $g=[F:\QQ]$. As before, associated unitary will be denoted groups $G$ and $H$. 

Let $\psi: \Ads /\QQ \rar \CC^\times$ be a non-trivial additive character,and $\chi: \Ads_E^\times/E^\times \rar \CC^\times$ a Hecke character such that $\chi|_{\Ads_F^\times} = \epsilon^m$ where 
$\epsilon: 
\Ads_F^\times/F^\times \rar \CC^\times$ is the quadratic character associated 
to $E/F$.

Let 
$$ \calS=\scrS(V(\Ads)^n)$$
be the space of Schwartz-Bruhat functions. We will occasionally identify it with $\scrS(Y\otimes_E V \otimes \Ads)$ using the fixed isomorphism 
$E^n\simeq Y= \Span_E\{e_1,\cdots, e_n\}.$

Associated to the data $(V,Q)$, $(W,J)$, $\psi$, and $\chi$, is a model of the Weil representation $\omega = \omega_{\psi, \chi}$ of 
$G(\Ads)\times H(\Ads)$ acting on $\calS$. The action of $h\in H(\Ads)$ on 
$\phi\in \calS$ is given by
\begin{align} (\omega(h)\phi)(x) = \phi(h^{-1} x), \ \ x \in V(\Ads)^n.
\end{align}

Let $w_0\in G(\QQ)=U(W,J)(F)$ be defined by
$$ w_0(e_i) = -f_i,\ \ \ w_0(f_j) = e_j.$$
The action of $G(\Ads)$ on $\calS$ is determined by
\begin{align*}
(\omega(m)\phi)(x) &= \chi(\det(a)) |\det(a)|_{\Ads_E}^{m/2} \phi(a^{-1} 
x), & & m=m(a)\in M(\Ads),\ a\in\GL_n(\Ads_E),\\
(\omega(n)\phi)(x) &= \psi(\tr( x,x)b) \phi(x),& & n=n(b)\in 
N(\Ads),\ b\in \Herm_n(\Ads_E),\\
(\omega(w_0) \phi)(x) &= (\calF_{Q} \phi)(x) = \int_{V(\Ads)^n} 
\phi(y) 
\psi((y,x)) dy,
\end{align*}
where $dy$ is the Haar measure with respect to the Fourier transform 
$\calF_{Q} = \calF_{Q,\psi}$ satisfies the Fourier inversion formula $\calF_\QQ( \calF_\QQ \phi)(x) = \phi(-x)$.

\subsection{Seesaw dual-reductive pairs} For $(V,Q)$ and $(W,J)$ as above, put
$$ \WW = W\otimes_E V,\ \ \   \langle\ , \ \rangle = \Tr_{E/F}(
J^\sigma \otimes_E Q).$$
Then $(\WW, \langle\ , \ \rangle)$ is a symplectic $F$-vector space. The 
$F$-groups $U(W,J)$ and $U(V,Q)$ form a \textit{dual reductive pair} in 
$\Sp(\WW)$. In other words, they are mutual centralizers under the 
natural map $U(W,J) \times U(V,Q) \hra \Sp(\WW)$. 
Applying $\Res^F_\QQ$ we obtain 
\begin{align}G\times H \hra \Res^{F}_\QQ \Sp(\WW).
\end{align}

Now assume that $(V,Q)=(V_0,Q_0)\otimes_{E_0}E$, and put
$$ W_0 = \Res_{E/E_0} W,\ \ \ J_0 = \Tr_{E/E_0} J.$$
Similar to $(\WW,\langle\ ,\ \rangle)$ define
\begin{align}\WW_0 = W_0\otimes_{E_0} V_0,\ \ 
\langle\ ,\ \rangle_0 = \Tr_{E_0/\QQ}(J_0^\sigma\otimes_{E_0} Q_0).\end{align}
Then $(\WW_0, \langle\ , \ \rangle_0)$ is a symplectic $\QQ$-vector space, and 
$G_0=U(W_0,J_0)$, $H_0=U(V_0,Q_0)$ are dual reductive 
pairs via the natural embedding
$$ G_0\times H_0 \hra \Sp(\WW_0).$$

It is easy to verify that there are canonical identifications
\begin{align}\WW_0 = \Res_{E/E_0} \WW,\ \ \ \langle\ ,\ \rangle_0 = 
\Tr_{F/\QQ}\langle\ ,\ \rangle.
\end{align}
There is a canonical homomorphism $\Res^F_\QQ \Sp(\WW) \rar \Sp(\WW_0)$, and 
the inclusions $H$, $G \subset \Res^F_\QQ \Sp(\WW)$, $H_0$, $G_0 \subset 
\Sp(\WW_0)$ are compatible with $\iota: H_0 \hra H$ and $\jmath: G \hra 
G_0$. In the terminology of \cite{K83}, $(G,H)$ and $(G_0,H_0)$ are 
\textit{seesaw 
	weakly dual pairs}, represented by a diagram
$$ \xymatrix { G_0 \ar@{-}[d] \ar@{-}[dr] &  H \ar@{-}[dl] \ar@{-}[d] \\
	G & H_0.}$$
Here the oblique lines indicate mutual centralizers and the vertical lines are 
inclusions.

Suppose $\psi_0: \Ads/\QQ \rar \CC^\times$ is a non-trivial additive character and $\chi_0: \Ads^\times_{E_0}/{E_0}^\times \rar \CC^\times$ be a Hecke character such that 
$\chi_0|_{\Ads^\times} = \epsilon_{E_0}^m$. 
where $\chi_{E_0}$ is the quadratic character associated to $E_0/\QQ$. Then 
again we have an associated Weil representation $\omega_0 = 
\omega_{\psi_0,\chi_0}$ of $G_0(\Ads)\times H_0(\Ads)$ acting on $\scrS(Y_0\otimes V_0(\Ads))$, where $Y_0 = \Res_{E/E_0} Y$ is a maximal isotropic subspace of $(W_0,J_0)$. 

As $E_0$-vector spaces we have 
$$V\otimes_E Y \cong (V_0 \otimes_{E_0} E)\otimes_E Y  \cong 
V_0\otimes_{E_0} Y_0$$
so that $\scrS(V_0\otimes_{E_0} Y_0)$ is in fact canonically identified with 
$\scrS(V\otimes_{E} Y)=S$. In order to obtain compatible Weil representations 
$\omega$ and $\omega_0$, we assume
\begin{align} \psi = \psi_0 \circ \Tr_{F/\QQ},\ \ \ \chi = \chi_0 \circ 
\mathrm{Nm}_{E/E_0}.
\end{align}
For instance $\psi_0$ and $\chi_0$ may be chosen first and $\psi$, $\chi$ defined as above.

Recall the $\ZZ$-bases $S=\{s_i\}_{i=1}^g$ and $R=\{r_i\}_{i=1}^g$ fixed for $\calO_F$ and 
$\frakd_F^{-1}$ in 
(\ref{bb}). For $b\in F$, let $B\in M_{g}(F)$ be the matrix of the $\QQ$-linear 
map $F\rar F,\ x\mapsto bx$, with respect to $S$ in the source, and 
$R$ in the target. Let $\xi = r_1 x_1 + \cdots r_g x_g\in V(F)$, with $x_i\in 
V_0$, and similarly $\eta = r_1 y_1 + \cdots r_g y_g$. Denote by $Q(x,y)\in M_g(E)$ the matrix with entries $Q(x_i,y_j)$.
\begin{lemma}\label{tracelem}$\tr_{E/E_0}(Q(\xi,\eta)b) = \Tr (Q(x,y) B)$.
\end{lemma}
\begin{proof}Left as linear algebra exercise.
\end{proof}

We fix the isomorphism 
\begin{align}\label{V0^nV} \scrS(V(\Ads)^n) \isomto \scrS(V_0(\Ads)^{ng}),\ \ \ 
\varphi \mapsto 
\varphi_0=\varphi\circ \beta_{V_0(\Ads)}^{n},\end{align}
with $\beta$ as in (\ref{bb}).

\begin{prop}The map $\varphi \mapsto \varphi_0$ intertwines the representations 
	$\omega_{Q,\psi,\chi}$ and $\omega_{Q_0,\psi_0,\chi_0}$ of $G$. In other words for 
	each $g\in G(\Ads)$ the following diagram commutes:
	$$ \xymatrix{ \scrS(V(\Ads)^n) \ar[d]_{\omega_{Q,\chi,\psi}(g)} 
		\ar[r]^{\varphi 
			\longmapsto \varphi_0} & 
		\scrS(V_0(\Ads)^{ng}) 
		\ar[d]^{\omega_{Q_0,\chi_0,\psi_0}(\jmath_0(g))}\\	\scrS(V(\Ads)^n) 
		\ar[r]_{\varphi \mapsto \varphi_0} & 
		\scrS(V_0(\Ads)^{ng}).}$$
\end{prop}

\begin{proof}
	It's enough to verify this for $n=1$, since taking direct sums implies the 
	statement for higher $n$.
	
	For $n= \left(\begin{array}{cc} 1 & b \\ & 1 \end{array}\right)$, $b\in 
	F$, 
	and $\phi\in \scrS(V(\Ads))$ we have
	$$ \omega_Q(n)\phi(\xi) = \psi(Q(\xi,\xi)b)\phi(x) = \psi_0(\tr_{F/\QQ} 
	Q(x,x)b)\phi(x) = \psi_0(\Tr(Q(x,x)\cdot B))\phi(x),$$
	the last identity by the lemma. The desired compatibility for $n\in N(\Ads)$ follows, since 
	$n$ 
	corresponds to $\left(\begin{array}{cc}1 & B \\ & 1 \end{array}\right)$ in $G_0(\Ads)$
	under the identification (\ref{RW}) of split skew-hermitian spaces 
	$E^2$ and $E_0^{2g}$ over $E_0$.

	The element $\omega_0\in G$ acts on $\scrS(V(\Ads))$ by Fourier 
	transform. We have
	$$ (\calF_{Q_0}\varphi_0) (y) = \int_{V_0(\Ads)^{g}} \varphi_0(x) 
	\psi_0(\tr Q_0(y,x))dx.$$
	For $\eta = \beta_\Ads(y)$, $\xi = \beta_\Ads(x)$, by the lemma we have
	$\psi_F (Q(\eta,\xi))= \psi_0( \tr Q_0(y,x))$. Changing variables by 
	$\beta_\Ads: 
	V_0(\Ads)^g \rar V(\Ads)$, we get
	
	$$ (\calF_{Q_0}\varphi_0)(\beta^{-1}(\eta)) = \int_{V(\Ads)} \varphi(\xi)  
	\psi(\tr Q(\eta,\xi)) C d\xi = C \cdot (\calF_{Q}\varphi)(\eta)$$
	where $C$ is determined by $({\beta_\Ads}^{-1})^*(dx) = C \cdot d\xi$. We 
	claim 
	that $C=1$, i.e. $(\beta_\Ads^*)(d\xi) = dx$.
	
	The product formula implies that the Haar measure on $V_0(\Ads)^g$ with 
	respect to which $\calF_{Q_0}$ is self-dual, in fact doesn't depend on 
	$Q_0$, 
	even though self-dual local factors do. This essentially comes down to  $|\det Q_0|_{\Ads_{E_0}}=1$. It follows that to check $C=1$ 
	we 
	can 
	assume $(V_0,Q_0)$ is the standard split hermitian form on $E_0^{m}$, and further that $m=1$.
	
	Write the self-dual measure $d\xi$ on $V(\Ads)=\Ads_{E}$ as a product $d\xi_{\infty} d\xi_0 = d\xi_\infty \prod_{v} d\xi_{v}$ of local 
	self-dual measures $d\xi_{v}$ on $E_v$. For each rational prime $p$, we have $\beta\otimes \ZZ_p=\beta_p: \ZZ_p^g 
	\rar 
	\prod_{v|p} \frakd_{F_v}^{-1}$. Since $d\xi_v$ is the self-dual measure on 
	$\calO_{F_v}$, $\beta_p^*(d\xi_v)$ 
	gives $\ZZ_p^g$ measure $|D_{F_v}|^{1/2}$. 
	Then $\beta^*(d\xi_0)$ gives $\wh{\ZZ}^g$ measure $|D_F|^{1/2}$. It follows that if $\{dx_{w_0}\}$ are self-dual measures on $V_0(E_{0,w_0})^g$ at all finite places $w_0$ of $E_0$, then $(\beta'_{\wh{\ZZ}})^*(\prod_w d\xi_{w}) = |D_F|^{1/2} 
	\prod_{w_0} dx_{w_0}$. On the other hand the self-dual measures 
	$d\xi_{\infty}$ and $dx_{\infty}$ on $E\otimes \RR$ and $E_0^g \otimes \RR$ both coincide with that on $\CC^{g}$. Then via $\beta_\infty: E_0\otimes \RR \rar E\otimes \RR$ we have $(\beta_\infty)^*(d\xi_{\infty}) = |\det \beta_\infty|
	dx_{\infty} = |D_F|^{-1/2}  dx_\infty.$
	It follows that
	$$ (\beta_\Ads)^*(d \xi) = |D_F|^{-1/2} dx_{\infty} \cdot |D_F|^{1/2} dx_f = 
	dx.$$
	
	The restriction map $\Gal(F^\ab/F) \rar \Gal(\QQ^\ab/\QQ)$ is surjective and induces an isomorphism $\Gal(E/F) \rar \Gal(E_0/\QQ)$. Under the reciprocity map of global class field theory, that restriction map is induced by the norm $\Nm_{F/\QQ}: \Ads_F^\times \rar \Ads^\times_\QQ$ which is the same as $\Nm_{E/E_0}$ restricted to $\Ads_F^\times$. This implies if $\chi_0 |_{\Ads^\times}=\epsilon_{E_0/\QQ}^m$, then $\chi_0\circ \Nm_{E/E_0}|_{\Ads_E^\times}=\epsilon_{E/F}^m$.  The compatibility of $\omega$ and $\omega_0$ for elements $m(a)\in G(\Ads)$ follows from this.
\end{proof}

\subsection{Fourier coefficients of theta integrals}

Using the Weil representation $\omega$ of $G(\Ads)\times H(\Ads)$ one associates to each $\phi\in 
\calS$ the theta kernel 
\begin{align}\Theta_\phi(g,h) = 
\sum_{x\in V(\QQ)^n} (\omega(g,h) \phi)(x),\ \ \ g\in G(\Ads),\ h\in H(\Ads).\end{align}
One may consider the theta integral  
\begin{align}
\int_{[H]} \Theta_{\phi}(g,h) dh,
\end{align}
where $dh$ is the invariant measure on $[H]=H(\QQ)\backslash H(\Ads)$ giving it total volume $1$. Let $r_V$ denote the Witt index of $(V,Q)$. By a theorem of Weil, the theta integral converges absolutely for all $\phi$ if and only if either $V$ is anisotropic, or
$$ \dim V - r_V > \frac{1}{2}\dim W +1.$$

Now we consider again the case $(V,Q)=(V_0,Q_0)\otimes_{E_0} E$, 
with $(V_0,Q_0)$ of signature $(n,1)$, and $(W_0,J_0)=\Res_{E/E_0} (W,J)$. For 
the pair of spaces $(V_0,Q_0)$, $(W_0,J_0)$, the numbers are
$$ \dim V_0  = n+1,\ \ r_{V_0} = 1,\ \ \dim W_0 = 2n.$$
These fall just outside the convergence range of Weil, so there is no guarantee 
that the theta integrals associated to $\omega_0$ converge. However, for the 
pair $(V,Q)$, $(W,J)$, we have 
$$ \dim V = n+1,\ \ \ r_{V} = 1,\ \ \dim W = 2,$$ 
which is within the convergence range since $n>1$. For $\phi\in \calS$, let 
$\phi_0$ denote its image in $\calS = \scrS(V(\Ads))$. The compatibility of 
the Weil representations $(\omega,\calS)$ and $(\omega_0,\calS_0)$ implies that
for $g\in G(\Ads)$, $h_0\in H_0(\Ads)$,
\begin{align}\label{TT} \Theta_{\phi_0}(\jmath(g),h_0) = 
\Theta_{\phi}(g,\iota(h_0))\end{align}
which implies 
$$ I(g;\phi) = \int_{[H_0]} \Theta_\phi(g,\iota(h_0)) dh_0$$
converges absolutely for any $\phi\in \calS$, $g\in G(\Ads)$. For convenience 
we now drop $\iota$ and $\phi_0$ from the notation, and simply write 
$\Theta_\phi(g,h_0)$, understood to mean either side of (\ref{TT}).

For $b\in F$, the $b$th Fourier coefficient of $I(g,\phi)$ with respect to 
the Siegel parabolic $P(\Ads)\subset G(\Ads)$ is by definition
$$ I_{b}(g;\phi) = \int_{N(\Ads)} I(ng;\phi) \ol{\psi(\tr_{E/F} bb_0)} dn,$$
where $n=n(b_0) = \left(\begin{array}{cc} 1 & b_0 \\ &1 \end{array}\right)\in 
N(\Ads).$

Now considering $\phi$ as an element of $\scrS(V(\Ads))$, we may write
\begin{align*} I_{b}(g;\phi) &= \int_{N(\Ads)} \int_{[H_0]}\sum_{\xi\in V(\QQ)} 
(\omega(ng;\phi)\phi)(\xi) \psi(-\tr_{E/F} bb_0)dh dn\\
&= \int_{[H_0]}\sum_{\xi\in V(\QQ)} \omega(g,h)\phi(\xi) \int_{N(\Ads)} \psi(\tr_{E/F}((Q(\xi,\xi)-b)b_0) dn dh.\end{align*}
The integral then picks out the terms $\xi\in V(\QQ)$ with $Q(\xi,\xi)=b$. As the volume of $N(\Ads)$ is 1, we obtain
$$ I_b(g,\phi) = \int_{[H_0]} \Theta_{b}(g,h;\phi) dh$$
where
$$ \Theta_b(g,h;\phi) =\sum_{\tiny \begin{array}{c}\xi\in V(\QQ)\\ Q(\xi,\xi)=b\end{array}} \omega(g)\phi(h^{-1}\xi).$$

Now assume that $\phi = \phi_\infty \otimes \phi_L$, where $\phi_L\in 
\scrS(V(\Ads_f))$ is the characteristic function of $L\otimes \Ads_f \subset 
V(\Ads_f)$, and that $\phi_\infty\in \scrS(V(\RR))$ is 
invariant under a maximal compact subgroup $K_0\subset H_0(\RR)$-invariant.  
Let $K_0(L)\subset H_0(\Ads_f)$ be the stabilizer 
of $L\otimes \Ads_f$, considered as a lattice in $V_0^n$, take 
representatives $h_1,\cdots h_l$ for 
$H_0(\QQ)\backslash H_0(\Ads_f)/K_0(L)$, and put
\begin{align} L_i = h_i(L\otimes \Ads_f) \cap V(\QQ),\ \ \ K_i = h_i K_0(L) 
h_i^{-1},\ \ \ \Gamma_i = K_i\cap H_0(\QQ).\end{align}

Then for $g_\infty\in G(\RR)$, 
$$ I_b(g_\infty,\phi) = \sum_{i=1}^l \int\limits_{\Gamma_i \backslash 
D_0} 
\sum_{\tiny 
\begin{array}{c}\xi\in L_i\ \ Q(\xi,\xi)=b\end{array}} 
\omega(g_\infty)\phi_\infty(h^{-1} \xi) dh,$$
where $dh$ is the invariant measure on $\Gamma_i\backslash D_0$ induced from 
$H_0(\RR)$.

Now let
$$I_{L_i,b} = \{\xi\in L_i: Q(\xi,\xi)=b\}.$$
For $\xi\in I_{L_i,b}$ we write $\calO_{\xi}$ for the $\Gamma$-orbit of $\xi$. 
Then for some $\xi_1,\cdots, \xi_t$, 
$$ I_b(g_\infty,\phi) = \sum_{i=1}^l \sum_{i=1}^t 
\int\limits_{\Gamma_i\backslash D_0} 
\sum_{\xi\in \calO_{\xi_t}} \omega(g_\infty)\phi_\infty(h^{-1}\xi)dh.$$
Now let $\Gamma_\xi$ denote the stabilizer of $\Gamma$ acting on $\xi$. Each 
orbit $\calO_{\xi}$ is then in bijection with $\Gamma/\Gamma_\xi$. Then 
$$ I_b(g_\infty;\phi)=\sum_{i=1}^l \sum_{j=1}^t \sum_{\gamma\in 
\Gamma_i\backslash\Gamma/\Gamma_\xi}\ \  \int\limits_{\Gamma_{\xi,i}  
\backslash 
D_0} \omega(g_\infty)\phi_\infty(h^{-1}\gamma \xi)dh$$
where $\Gamma_{\xi,i} = \Gamma_\xi \cap \Gamma_i$. Now put
$$C_{\xi,i} = \Gamma_{\xi,i}\backslash D_{\xi,i},\ \ \ M_i = \Gamma_i 
\backslash D_0,\ \ E_i = \Gamma_{\xi,i}\backslash D_0.$$
Then $C_{\xi,i}$ is a special cycle on $M_i$ via a map
$$ \iota_{\xi,i}: C_{\xi,i}  \rar M_i$$
induced by $D_\xi \subset D_0$. The natural inclusion $D_{\xi,i} \hra D_0$ 
induces a fibration $D_0 \thrar D_{\xi,i}$ by geodesics, which provides a 
fibration $\pr_i: E_i \rar C_{\xi,i}$. If $f$ is an integrable function on 
$E_i$, we may compute the integral using Fubini's theorem by first integrating 
over the fibers of $\pr_i$ to get $(\pr_i)_* f$, then integrating that on 
$C_{\xi,i}$. Now suppose $f=f_{g,\xi}$ is the function on $D_0$ induced by 
$h\mapsto \omega(g)\phi_\phi(h^{-1}\xi)$. If $\kappa_{g,\xi}=(\pr_{i})_* 
f_{g,\xi}$ is \textit{constant} on $D_{\xi,i}$, we obtain 
$$ I_b(g_\infty;\phi)=\sum_{i=1}^l \sum_{j=1}^t \sum_{\gamma\in 	
\Gamma_i\backslash\Gamma/\Gamma_\xi}\ \  \vol(C_{\xi,i}) 
\kappa_{g_\infty,\gamma\xi}.$$
The Kudla-Millson Schwartz function $\phi_\KM\in \scrS(V(\RR))$ has this 
property and the corresponding constant, calculated in \cite{KM87}, is 
determined by
\begin{align}\label{kxi} \kappa(\xi) = \kappa_{1,\xi} = i^{-n} \exp(-\pi 
\tr_{E/F} Q(\xi,\xi)).\end{align}

Let $\tau=(\tau_i)_i \in \frakH^n$, write $\tau_i = u_i + i v_i$, and put $a_i 
= \sqrt{v_i}$. Let $\alpha\in \GL_1(E\otimes \RR)$ and $\upsilon\in 
\Herm_1(F\otimes\RR)$ correspond to $a=(a_i)_i$ and $u=(u_i)_i$, and put
$$ g_\tau =\left(\begin{array}{cc} 1 & u \\ & 1\end{array}\right) 
\left(\begin{array}{cc} a& \\ & {}^t \ol{a}^{-1}\end{array}\right)\in G(\RR).$$
Then for $\phi\in \scrS(V(\RR))$, by definition
$$ \omega(g_\tau)\phi(\xi) = \chi(\alpha) |\alpha|_{\Ads_{E}}^{m/2} 
\phi(\xi\alpha) \psi(\tr_{E/F} 
Q(\xi ,\xi)\upsilon).$$ 
We have $\chi(a) = 1$, since the assumption 
$\chi_{\Ads_{F}^\times}=\epsilon^m_{E/F}$ implies $\chi$ has trivial infinity 
type. Also
$$|\alpha|_{\Ads_{E}}=|\Nm(\det(\alpha))|_{\Ads_{E_0}} = \prod_i 
|a_i|^2 = \prod_i v_i,$$
and
$$ \psi ( Q(\xi,\xi)\nu) = \psi_0(\tr_{E/E_0} Q(\xi,\xi)\nu) = \psi_0(\Tr \beta 
u) = \exp(2\pi i \sum_j \lambda_j(b) u_j) = e_*(bu),$$
where $u$ is considered as a column vector.
Therefore
$$ \omega(g_\tau)\phi_\infty(\xi) = \phi_\infty(h^{-1}\xi a)\prod_i \Im(\tau_i) 
^{m/2}  
e_*(bu).$$
For $\phi=\phi_\KM$, it follows from (\ref{kxi}) that
$$ \int_{E_{\xi,i}} \omega(g_\tau) 
\phi_\KM(h^{-1}\xi)dh = \vol(C_{\xi,i})\prod_{i} \Im(\tau_i)^{m/2} e_*(bu) 
i^{-n} \exp(-\pi \tr_{E/F}Q(\xi a,\xi a)).$$
Now 
$$\exp(-\pi \tr_{E/F} Q(\xi a,\xi a)) = \exp(2\pi i \sum_{j}  \lambda_j(b)) 
i v_j ) = e_*(biv)$$
so in fact
$$ \int_{E_{\xi,i}} \omega(g_\tau) 
\phi_\KM(h^{-1}\xi)dh = \prod_{i} \Im(\tau_i)^{m/2} 
i^{-n} \vol(C_{\xi,i}) e_*(b\tau).$$

We then obtain the formula
\begin{align}\label{Ibform} I_b(g_\tau;\phi_\KM)=i^{-n} \prod_j|\Im \tau_j|^{m/2}\sum_{i=1}^l 
\sum_{j=1}^t 
\sum_{\gamma\in 	
	\Gamma_i\backslash\Gamma/\Gamma_{\xi_t}}\ \   
	\vol(C_{{\xi_t},i}) e_*(b \tau).\end{align}

For $\tau\in \frakH^n$, put
$$ F_i(\tau) = \sum_{b\in F} I(C_i,C_b) e_*(b\tau)$$
where $C_i = D_0/\Gamma_i$, and $I(C_i,C_b)$ is defined analogously to $I(C_0,C_b)$. Put
$$ F(\tau) = \sum_{i=1}^l F_i(\tau).$$
We have thus proven:

\begin{thm}$F(\tau)=i^n |\Im \tau_i|^{-m/2} I(g_\tau;\phi_\KM)$. In particular, $F(\tau)$ is a modular form of weight $\frac{m}{2}$. 
\end{thm}
\begin{proof}From (\ref{Ibform}) we get
	$$ I(g_\tau;\phi_\KM) = \sum_{b\in f} i^{-n} \prod_j|\Im \tau_j|^{m/2}\sum_{i=1}^l 
	\sum_{j=1}^t 
	\sum_{\gamma\in 	
		\Gamma_i\backslash\Gamma/\Gamma_{\xi_t}}\ \   
	\vol(C_{\xi,i}) e_*(b \tau).$$
	
By Theorem \ref{C0Cb} and the discussion preceding it, we have
$$ \sum_{j=1}^t \sum_{\gamma\in \Gamma_i \backslash \Gamma/\Gamma_{\xi_t}} \vol(C_{\xi_t,i})= I(C_i, C_b).$$
so that 
$$ I(g_\tau; \phi_KM) = \sum_{b\in F} I_b(g_\tau,\phi_\KM) = i^{-n} \prod_j |\Im \tau_j|^{m/2} \sum_{i=1}^l \sum_{b\in F} I(C_i,C_b)e_*(b\tau).$$
\end{proof}

In the second part of the article we show that $I(g;\phi_\KM)$, and hence $F(\tau)$, can be identified with the restriction of automorphic forms on $U(n,n)$ related to Siegel Eisenstein series. 

\section{Restrictions of Eisenstein Series}

We recall the construction of Eisenstein series associated to sections of parabolically induced representations. The Siegel-Weil formula identifies theta integrals in Weil's convergence range with values of such Eisenstein series at particular points. Outside this range the regularized-Siegel Weil formula gives an analogous statement obtained by regularizing the theta integral.

\subsection{Eisenstein series}

We extend the Hecke character 
$\chi:\Ads_{E}^\times/E^\times \rar \CC$ to a character of 
$P(\Ads)$ by 
\begin{align}\chi \left(\begin{array}{cc}a & * \\ & {}^t 
\ol{a}^{-1}\end{array}\right) = \chi(\det a).\end{align}

We also have the determinant character $p(a)\in P(\Ads) \mapsto 
|\det(a)|^s$, for $s\in \CC$. There is then a normalized induced representation 
\begin{align} I(s,\chi) = \Ind_{P(\Ads)}^{G(\Ads)} \chi |\det(a)|^s. \end{align}

A \textit{section} of $I(s,\chi)$ is a function $f^{(s)}: G(\Ads) \rar \CC$ 
depending on 
a parameter $s\in \CC$ such that 

\begin{align} f^{(s)}(pg) = \chi(a) \det|a|^{s+\frac{n}{s}} f(g), \ \ \ \text{ 
	for }g\in 
G(\Ads), p\in 
P(\Ads).
\end{align}

The section $f^{(s)}$ is called holomorphic if it's $K$-finite and holomorphic 
in the variable $s$. A \textit{standard section} is a holomorphic section whose 
restriction to $K$ is independent of $s$.

Put
\begin{align} s_0 = \frac{m-n}{2}.\end{align}

The Siegel Eisenstein series associated to a holomorphic section $f^{(s)}$ of 
$I(s,\chi)$ is the series

\begin{align} E(g,f^{(s)})=\sum_{\gamma \in P(F)\backslash G(F)} f^{(s)}(\gamma 
g).\end{align}

It is convergent for $s\geq n/2$, and has meromorphic continuation to the 
$s$-plane.

Associated to each $\varphi\in \scrS(V(\Ads)^n)$ is a standard 
\textit{Siegel-Weil section} $f_{\varphi}^{(s)}$ given by
\begin{align}\label{stdSW} f_\varphi^{(s)}(g) = |\det a(g)|_{\Ads_E}^{s-s_0} 
\omega(g) 
\varphi(0).\end{align}

The Eisenstein series $E(g,f_\varphi^{(s)})$ is holomorphic for 
$\Re(s)\geq 0$, except at possibly 
at $s=s_0$, where it has a pole of order at most $1$. Its Laurent 
expansion is written
\begin{align}\label{Esphi} E(g,f_{\varphi}^{(s)}) = 
\frac{A_{-1}(\varphi)}{s-s_0} + 
A_0(\varphi) + 
A_1(\varphi)\cdot (s-s_0)+ \cdots .
\end{align}

By a result of Weil, the theta integral converges absolutely for all 
$\varphi\in 
\scrS(V^n(\Ads))$ if and only if either $r=0$ or $m-n>r$. If furthermore 
$m>2n$, the Siegel-Weil formula holds:
$$E(g,f_{\varphi}^{(s_0)})=I_{Q}(g,\varphi).$$

\subsection{Regularized theta integrals}

Beginning with the work of \cite{KRa88}, the Siegel-Weil formula has been 
extended outside the Weil range. For this the theta integral on the right hand side of the formula must be regularized.

An ingredient used in the regularization is the auxiliary Eisenstein series 
$E_H(s,h)$ on 
$H(\Ads)$ defined by
\begin{align} E_{H}(s,h) = \sum_{\gamma \in P_H(F)\backslash H(F)} 
f_0^{(s)}(\gamma h),\end{align}
where
$f_0^{(s)}$ is the standard $K_H$-spherical section determined by 
$f_0^{(s)}(1)=1$. $E_{H}(s,h)$ has a pole of order at most 1 at $s=\rho_H$, 
where
\begin{align}\rho_H = \frac{m-r}{2},
\end{align}
and the residue 
\begin{align} \kappa = \Res_{s=\rho_H} E_H(s,h)\end{align}
is a constant function. 

The regularization in \cite{KRa88} is achieved using a particular element 
$z_G$ in 
the central enveloping algebra of 
$\Lie(G(\RR))_\CC$, such that
$\Theta(\omega(z)\varphi;g,h)$ is rapidly decreasing on $H$ for all $\varphi\in 
\scrS(V^n(\Ads))$ \cite{KRa88,GQT}. By Howe duality one may choose a 
corresponding element $z_H$ 
in the universal enveloping algebra of $\Lie(H(\RR))_\CC$, such that 
$\omega(z_G)$ and $\omega(z_H)$ coincide as operators on $\scrS(V(\Ads)^n)$. 
Then one has
\begin{align} \omega(z_H) \cdot E_H(s,h) = P(s)\cdot E_h(s,h),
\end{align}
for an explicitly computable function $P(s)$ that depends in general on the 
choice of $z_H$.

The regularized theta integral is then defined to be
\begin{align}\label{B}
B(s,\varphi) = \frac{1}{\kappa \cdot P(s)} \cdot \int_{[H]} 
\Theta(\omega(z)\varphi;g,h) E_{H}(s,h)dh.
\end{align}

We will assume the parameters $r,m,n$ satisfy
\begin{align}\label{pars} 0 < m-n \leq r \leq n.\end{align}
In that case $P(s)$ has a simple zero at $s=\rho_H$, and $B(s,\varphi)$ has a 
pole of 
order at most $2$ \cite[Lemma 3.8]{GQT}. By convention, the Laurent expansion 
of $B(s,\varphi)$ at $s=\rho_H$ is written
\begin{align}\label{LaurentB} B(s,\varphi) = 
\frac{B_{-2}(\varphi)}{(s-\rho_H)^2} + 
\frac{B_{-1}(\varphi)}{(s-\rho_H)}+B_0(\varphi)+\cdots\ .\end{align}

\subsection{Regularized Siegel-Weil formula}

Since $m-n>0$ by the assumption (\ref{pars}), the parameters fall in the 
so-called second-term range. In this case the regularized Siegel-Weil formulas 
of \cite{Ich04MZ} and \cite{GQT} apply, which we now briefly state.

Let $V''\subset V$ be a split hermitian subspace of dimension $2(m-n)$, and 
denote its orthogonal complement by $V'$. Put $Q'=Q|_{V'}$ and
\begin{align*}H'=U(V',Q').\end{align*}

Writing $V''$ as a direct sum $X \oplus X^*$ of (dual) isotropic subspaces of 
dimension $r_0=m-n$, we have $V=X \oplus V' \oplus X^*$. Using this 
decomposition, we define the parabolic subgroup $P_H\subset H$ by

\begin{align} P_H(R) = \left\{ \left(\begin{array}{ccc} a  & * & * \\ & g & 
* \\ & & 
{}^t \ol{a}^{-1}\end{array}\right): a\in \GL_{r_0}(R\otimes X),\ g\in 
H'(R)\right\},\ \ \ R\in \QQ\dash \Alg.
\end{align}

Fix a good maximal compact subgroup $K\subset H(\Ads)$ such that $H(\Ads) = 
P_H(\Ads)\cdot K$. Let $dk$ be the Haar measure on $K$ that gives it volume 
1, and define
$$ \pi_K: \scrS(V^n(\Ads)) \rar \scrS(V^n(\Ads)),\ \ \ \pi_K(\phi)(x) = 
\int_{K} \phi(kx)dx.$$
The \textit{Ikeda map} is
$$ \Ik: \scrS(V^n(\Ads)) \rar \scrS(V'^n(\Ads)),\ \ \ \Ik(\varphi)(x) = 
\int_{X^n_{\Ads}} \phi\left(\begin{array}{c} x \\ y \\ 0 \end{array}\right)dy,$$
where the coordinates inside the integral correspond to $V^n = X^n \oplus V'^n 
\oplus X^{*n}.$

The regularized Siegel-Weil formula of \cite{Ich04MZ} then states
\begin{align}\label{Ich}
A_{-1}(\phi) = c_K I_{Q'}(g,\Ik\circ \pi_K(\phi)).
\end{align}

Furthermore, by the main results of \cite[Theorem 8.1]{GQT} 
\begin{flalign}
& & A_{-1}(\phi) &= B_{-2}(\phi)& & & &(\text{first-term 
identity})\label{first}\\
& & A_{0}(\phi) &= B_{-1}(\phi) - \kappa' \left\{ 
B_0'(\Ik(\pi_{K_H}(\phi)))\right\}\ (\mod \Im A_{-1}). & & & &
(\text{second-term identity})\label{second}
\end{flalign}

Here $\kappa'$ is an explicit constant, and $B_0'$ is the (possibly 
regularized) theta integral associated to the dual reductive pair $(G,H')$. If $V'$ is anisotropic, the term in $\{\cdots \}$ is to be interpreted as zero. 
We note the second-term identity is modulo the entire image of the map $A_{-1}$ 
from $\scrS(V(\Ads)^n)$ to automorphic forms on $G(\Ads)$.

\subsection{Refined regularization} The second-term identity simplifies when $\phi\in \calS$ satisfies certain conditions, and $m=n+1$, $r=1$. Note these values fall in the second-term range. By (\ref{second}) the value of
$E(g,f_{\varphi}^{(s)})$ at $s=s_0$ is in that case
$$ A_0(\phi) = B_{-1}(\phi)\ (\mod \Im A_{-1}).$$

The specific conditions we impose on $\phi$ are the following:
\begin{enumerate}
	\item $\phi$ is factorizable: $\phi=\otimes_{v} \phi_v$, for $\phi_v\in 
	\scrS(V(\QQ_v)^n)$.
	\item For some place $v_0$, $\Ik_{v_0}\circ \pi_{K_{v_0}} (\phi_{v_0})=0$, where 
	$\Ik_{v_0}$ and 
	$\pi_{K_{v_0}}$ are the local components of $\Ik$ and 
	$\pi_K$.
\end{enumerate}
In particular $(2)$ implies $\Ik\circ \pi_K (\phi)=0$. By (\ref{Ich}) and (\ref{first}) then
$$ B_{-2}(\phi) = A_{-1}(\phi)=0.$$
In other words, $B(\phi)$ has a pole of order at most $1$ at 
$s=s_0$.

Recall the global Siegel principal series representations $I_{n}(s,\chi)$ 
of 
$G_{n}(\Ads)$ and the $H(\Ads)$-invariant function
$$ \Phi_n^{(s)}: \scrS(V(\Ads)^n) \rar I_{n}(s,\chi),\ \ \ 
\Phi_n^{(w)}(g) = 
(\omega_{n}(g)\psi)(0).$$
There's a factorization $\Phi_n^{(s)} = \otimes_v \Phi_{n,v}^{(s)}$ into local components. We put 
\begin{align}\Phi_n = \Phi_n^{(s_0)}.\end{align} Then $\Phi_n$ is $G(\Ads)$-intertwining and 
maps $\psi\in 
\scrS(V(\Ads)^n)$ to $f_{\psi}^{(s_0)}$, the standard Siegel-Weil section 
introduced before.

There are standard global intertwining operators 
$$ M_n(s,\chi): I_n(s,\chi) \rar I_n(-s,\chi),\ \ \ (M_n(s,\chi)f)(g)=\int_{N(\Ads)} f(wng) dn$$
which factor into a tensor product of normalized local intertwining 
operators 
\begin{align}M_n(s,\chi) = \otimes_v M_{n,v}^*(s,\chi).
\end{align}
The normalization we use is the one given by Lapid-Rallis \cite{LapRal}, which is the same as in \cite{GQT}, \cite{GanIchino}. It's known that the derivative 
$M_{n,v}^*{}'(0,\chi)$ 
commutes with $M_{n,v}^*(0,\chi)$ and 
preserves the irreducible submodules of $I_{n,v}(0,\chi)$. The derivative $M'(0,\chi)$ appears in the regularized Siegel-Weil formula via the following construction. 

Choose and fix a factorizable $\phi_1 = \otimes_{v} \phi_{1,v} \in 
\scrS(V(\Ads))$ such that
\begin{align} \phi_1(0)=1,\ \ \ \pi_K (\phi_{1}) = \phi_1.\end{align}
For $\phi\in \scrS(V(\Ads)^n)$, consider $\phi_1\otimes \phi$ as an element of $\scrS(V(\Ads)^{n+1})$. Then there exists $\phi_M \in \scrS(V(\Ads)^{n+1})$ such that
\begin{align} \label{phiM} \Phi_{n+1} (\phi_M) = M_{n+1}'(0,\chi) 
\Phi_{n+1}(\phi_1 \otimes \phi).\end{align}
Let $\phi' \in \scrS(V(\Ads)^n)$ be the restriction of $\phi_M$ to $V(\Ads)^n$ 
via 
$$\phi'(x) = \phi_M(0,x).$$

\begin{prop}\label{refined}For $\phi\in \scrS(V(\Ads)^n)$ satisfying $(1),(2)$, and $\phi'$ as above,we have
	$$ A_0(\phi) - \frac{1}{2} A_{-1}(\phi')= B_{-1}(\phi).$$
\end{prop}
\begin{proof}

Since $n-m=1$, the pair of spaces $(W_{n},J_n)$, $(V,Q)$ lie in the so-called 
\textit{boundary range}, i.e. $n$ is the smallest value outside Weil's 
convergence. Then $G_{n+1}(\Ads)$, $H(\Ads)$ are inside the classical convergence range, and so 
\begin{align}A_0(\phi_1\otimes \phi) = 2 B_{-1}(\phi_1\otimes \phi)
\end{align}
for $\phi_1\otimes \phi\in \scrS(V(\Ads)^{n+1})$.
The second-term identity (\ref{second}) of \cite{GQT} is obtained for the boundary case by computing the constant terms of both sides of the above with respect to a parabolic 
subgroup of $G_{n+1}$ with Levi component $\GL_1\times G_n$, then 
taking 
residues 
at $s=0$. In fact following through their computations we see the precise identity is 
\begin{align}\label{raw} A_0(\phi) - \frac{1}{2} A_{-1}(\phi') = 
B_{-1}(\phi) - 
C_r' \cdot B_0(\Ik\circ \pi_{K_H} (\phi)) + C_r \cdot B_{-1}(\Ik\circ 
\pi_{K_H} 
(\phi))\end{align}
for some explicit constants $C_r$, $C'_r$.

The statement then follows from the assumption $\Ik\circ \pi_K(\phi)=0$.
\end{proof}

In the next section, we will show that under the same assumptions the (unregularized) theta integral $I(g;\phi)$ in fact converges, and is equal to $B_{-1}(\phi)$.
\subsection{The mixed model}

For the moment let us again allow any general pair $(V,Q)$, $(W,J)$, where $(V,Q)$ has Witt index $r$. Recall the decomposition $V = X \oplus V_\an \oplus X^*$ from (\ref{Van}) and 
$W = Y \oplus Y^*$. The model of the Weil representation of $\Sp(\WW)(\Ads)$ 
described before acts on $\calS = \scrS(Y\otimes_E V(\Ads))$. There 
is a \textit{mixed model} of the same representation where the space acted on is $\wh{\calS} = 
\scrS((Y\otimes_E V_\an + W\otimes_E X^*)(\Ads))$. We have
$$ Y\otimes_E V  = Y\otimes_E X + Y\otimes_E V_\an + Y\otimes_E X^*$$
and 
$$Y\otimes_E V_\an + W \otimes_E X^* = Y\otimes_E V_\an + Y\otimes_E X^* + 
(Y\otimes_E X)^*$$
where $Y^*\otimes_E X^*$ has been identified with $(Y\otimes_E X)^*$. In these 
coordinates the intertwining map $\calS \rar \wh{\calS},\ \phi \mapsto 
\wh{\phi}$ is 
given by the partial Fourier 
transform
\begin{align}\label{partFou} \wh\phi (v_0, u', v') = \int\limits_{(Y\otimes_E 
	X)(\Ads)} 
\phi\left(\begin{array}{c}u \\ v_0 \\ u'\end{array}\right) \psi_F (v'(u)) 
du.\end{align}
Again for convenience we make the identifications 
$$Y\otimes_E V_\an = V_\an^n,\ \ \ W\otimes_E X^* = W^r$$ 
so that $\wh\calS = \scrS(V_\an(\Ads)^n \oplus W(\Ads)^r)$. Then the theta kernel associated to $\phi\in \calS$ may be written as
\begin{align}\label{thetamixed} \Theta(g,h;\phi) = \sum_{v_0\in V_\an(F)^n,\ w\in W(F)^r} 
\omega(g,h)\wh{\phi}(v_0,w).\end{align}

Let $\omega(z)$ be the operator on $\calS$ corresponding to the regularizing 
elements $z_G$ and $z_H$. For the moment we identify $W(F)^r$ with the matrix 
group $M_{n,r}(F)$. Two essential facts that enable the regularization procedure are that $\omega(z)\phi(v_0,w)=0$ if $\rank(w)<r$, and 
that for all $\phi\in \calS$, the function
$$ \Theta'(g,h,\phi)=\sum_{\tiny \begin{array}{c} v_0\in V_\an(F)^n,\ 
	w\in 
	M_{n,r}(F)\\ 
	\rank(w)=r\end{array}} 
\omega(g,h)\wh{\phi}(v_0,w)$$
is rapidly decreasing on $H(F)\backslash H(\Ads)$.

\begin{prop}\label{rapid}Assume $r=1$. If $\phi\in \scrS(V(\Ads)^n)$ is 
	$K_H$-invariant and satisfies $\Ik(\phi)=0$, then $h \mapsto \Theta(g,h;\phi)$ 
	is rapidly decreasing on $H(F)\backslash H(\Ads)$.
\end{prop}

\begin{proof}
	
	We show that in the expansion (\ref{thetamixed}) all the terms with $\rank(w)<r$ vanish. 
	From (\ref{partFou}), it's easy to see that 
	$$ \Ik(\phi)(v_0) = \wh{\phi}(v_0,0).$$
	
	On the other hand if $r=1$, the condition $\rank(w)=r$ is simply $w\neq 0$, so that
	$\wh{\phi}(v_0,0) = 0$ for all $v_0\in V_{an}(F)^n$. Now for $h\in H(\Ads)$ write 
	$h=pk$ with $k\in K_H$ and 
	$$ p^{-1} = \left(\begin{array}{ccc} a & r & s \\ & h_0 & t \\ & & {}^t 
	\ol{a}^{-1} 
	\end{array}\right)\in P_H(\Ads).$$
	Since $\phi$ is $K_H$-invariant, 
	$$ \omega(h) \phi \left(\begin{array}{c} u\\ v_0 \\ 0\end{array}\right) = 
	\omega(p)\phi\left(\begin{array}{c} u\\ v_0 \\ 0\end{array}\right)  = 
	\phi\left(\begin{array}{c} au + rv_0\\ h_0 v_0 \\ 0\end{array}\right)$$
	so that $\omega(h)\wh{\phi}(v_0,0)= \Ik (\phi)(h_0 v_0)=0$ for all $h\in 
	H(\Ads)$. From the fact that the Ikeda map is $G(\Ads)$-intertwining, it also 
	follows that 
	$$(\omega(g)\wh{\phi})(v_0,0)=\wh{\omega(g)\phi}(v_0,0)=\Ik(\omega(g)\phi)(v_0)=\omega(g)\Ik(\phi)(v_0)=0.$$ Then we have $\omega(g,h)\wh{\phi}(v_0,0)=0$ for all $v_0$, so that
	$$ \Theta(g,h;\phi) = \Theta'(g,h;\phi).$$
	Since the right-hand side is rapidly decreasing, the proposition follows.
\end{proof}

Now let us restrict again to the case where $(V,Q)=(V_0,Q_0)\otimes_{E_0} E$, with $(V_0,Q_0)$ having signature $(n,1)$, $n=[F:\QQ]$, and $(W_0,J_0)=(\Res_{E/E_0} W_1, \tr_{E/E_0} J_1)\simeq (W_n,J_n)$. If $\phi\in \calS$ is $K_{H_0}$-invariant, and $\Ik(\phi)=0$, by the proposition the theta integral
$$ I(g;\phi)=\int_{[H_0]} \Theta(g,h;\phi) dh,\ \ \ g\in G_0(\Ads),$$
converges, even though $(V_0,Q_0)$, $(W_0,J_0)$ are outside Weil's convergence range. On the other hand, so does the regularized theta integral
$$I^\reg(g,s;\phi)=\int_{[H_0]} \Theta(g,h;\omega(z)\phi) E_{H_0}(h,s) dh$$
for large enough $\Re(s)$, and it has meromorphic continuation to the $s$-plane.  

\begin{prop}Suppose that $r=1$, $\phi$ is $K_{H_0}$-invariant, and $\Ik(\phi)=0$. For $\rho = \frac{n}{2}$, we have
	$$ \Res\limits_{s=\rho} \frac{1}{P(s)\kappa } I^\reg(g,s,\phi) = I(g;\phi).$$
\end{prop}
\begin{proof}By Proposition (\ref{rapid}) $\Theta(g,h;\phi)$ is rapidly decreasing on $[H_0]$, so we can write
	\begin{align}\int_{[H_0]} \Theta(g,h;\omega(z)\phi) E_{H_0}(h,s) dh &= \int_{[H_0]} \Theta(g,h;\phi) (z_{H_0}^*\cdot E_{H_0}(h,s)) dh\\
	& = P(s) \int_{[H_0]} \Theta(g,h;\phi) E_{H_0}(h,s) dh.
	\end{align}
	Then 
	\begin{align*} \Res\limits_{s=\rho} \frac{1}{P(s)\kappa } I^\reg(g,s,\phi) &= \Res_{s=\rho} \frac{1}{\kappa} \int_{[H_0]} \Theta(g,h;\phi) E_{H_0}(h,s) dh\\
	&= \frac{1}{\kappa} \int_{[H_0]} \Theta(g,h;\phi) \left(\Res_{s=\rho} E_{H_0}(h,s) \right) dh\\ 
	&= \int_{[H_0]} \Theta(g,h;\phi) dh.
	\end{align*}
\end{proof}
\begin{cor}Suppose $\phi$ satisfies the assumptions of Proposition \ref{refined}. Then
	$$ I(g;\phi)=E(g;f_{\phi}^{(s_0)}) - \frac{1}{2} \Res\limits_{s=s_0} E'(g,f_{\phi'}^{(s)}).$$
\end{cor}
\begin{proof}
\end{proof}

Now we apply the above to $\phi=\varphi_\KM \otimes \varphi_L \in \scrS(V_0(\Ads)^n)$, which was used previously in $\S 3$ to prove the modularity of 
$$F(\tau) = \sum_{i=1}^l F_i(\tau) = \sum_{i=1}^l \sum_{b\in F} I(C_i,C_b) e_*(b\tau).$$
By construction, $\phi$ is factorizable. Theorem \ref{Ikedakills} from the appendix will show that $\phi_\KM$ is killed by the Ikeda map. Then the corollary therefore applies to $\phi$.

Let $T =\in \CC\frakH_n$ be an element of the hermitian upper-half plane of degree $n$, so that $T=U+iV$ with $U$, $V\in \Herm_n(\CC)$, and $V$ positive-definite. Write $T = {}^t \ol{A} A$, and put
$$ g_T = \left(\begin{array}{cc} 1 & U \\ & 1 \end{array}\right)\left(\begin{array}{cc} A & \\ & {}^t \ol{A}^{-1} \end{array}\right)\in G_0(\RR),\ \ \ \wt{F}(T) = i^{n}|\det T|^{-m/2} I(g_T;\phi).$$

Then $F(T)$ is a hermitian modular form of degree $n$, and essentially a Siegel Eisenstein series. The map $G(\RR) \hra G_0(\RR)$ then induces an injection $\iota_\RR: \frakH^n \rar \CC\frakH_n$. We have thus proved:

\begin{thm}For $\phi = \varphi_\KM \otimes \phi_L$, 
	$$\wt{F}(T) = E(g_T;f_{\phi}^{(\frac{1}{2})}) - \frac{1}{2} \Res_{s=\frac{1}{2}} E'(g_T,f_{\phi'}^{(s)}).$$
	
	In particular, $F(\tau)$ is the restriction of the right-hand side above to $\frakH^n \subset \CC\frakH_n$.
\end{thm}

\subsection*{Remarks}
The same argument as in the computation of $I_b(g,\phi)$, identifies the \textit{non-degenerate} Fourier coefficients $I(g,\phi)$, taken with respect to $P_0(\Ads)\subset G_0(\Ads)$, with intersection numbers $I(C_\beta,C_0)$, i.e. those indexed by $\beta\in \Herm_n(E_0)$ with $\det(\beta)\neq 0$. The restriction map to $\frakH^n$ then groups together all such coefficients $\beta$ with ${}^t \rho \beta \rho = b$. 

It's possible, but we have not been able to show, that the term $\Res_{s=s_0} E'(g,f_{\phi'}^{(s)})$ occurring in the refined regularization vanishes. This would be the case for instance, if the operation $\phi \mapsto \phi'$ preserves the kernel of the Ikeda map. In that case the refined formula would say that the Siegel-Weil formula in fact holds outside the convergence range for particular $\phi$, such as $\varphi_\KM\otimes \varphi_L$.

The ``generalized'' intersection volumes may have geometric interpretations in terms of the spectacle cycles of Funke and Millson \cite{FunkMill11}.

\section{Appendix}

\subsection{The Kudla-Millson construction}

First we recall the definition of the Kudla-Millson Schwartz function 
$\varphi_{q,q}^+$, following \cite{KM86, 
	KM87, KM90IHES} closely. We define an adelic Schwartz function 
$\varphi_{\mathrm{KM}}\in \scrS(V(\Ads)^n)$ using $\varphi_{q,q}^+$, and we 
show it is annihilated by a regularizing Ikeda map 
$$\Ik\pi^{Q}_{Q'}: \scrS(V(\Ads)^n) \rar \scrS(V'(\Ads))^n,$$
for some complementary subspace $V'$ of $V$. That will imply the corresponding 
Eisenstein series $E(s,\varphi_{\mathrm{KM}})$ has a simple pole at 
$\rho_H=\frac{m-r}{2}$.

Let $(U,(\ , \ ))$ be a non-degenerate complex hermitian space of dimension $m$ 
and signature 
$(p,q)$, with 
$pq\neq 0$ and $p<q$, to be thought of as one of the archimedean completions of 
$(V,Q)$ 
from the introduction. Put
\begin{align}G_U=\SU(U,\ (\ ,\ ))
\end{align}

Let $D_0$ be the set of negative-definite $q$-dimensional subspaces of $U$,
$Z_0\in D_0$ the span of $u_{p+1},\cdots, u_m$.  It may be identified with the 
hermitian symmetric domain associated to $G_U$ as follows.

Fix a basis $u_1,\cdots, u_m$ with dual $z_1,\cdots, 
z_m$ with respect to which 
$(\ ,\ )$ has the standard form
$$ (u,u)=\sum_{i=1}^p |z_i|^2 - \sum_{j=p+1}^m |z_j|^2,\ \ \ u \in U.$$
Let $Z_0 = \Span\{u_{p+1},\cdots, u_m\}\in D_0$. Then $D_0\simeq G_U/K$, 
where $K$ is the maximal compact subgroup stabilizing 
$Z_0$.

Let $\frakk = \Lie(K)_\CC$ and $\frakg  = \Lie(G_U)_\CC$ be the complexified 
Lie algebras, and $\frakg = \frakk + 
\frakp$ be the orthogonal decomposition with respect to the Killing 
form on $\frakg$. Then $\frakp$ can be identified with the complexified tangent 
space $T_{Z_0}(D_0)$, and its complex dual $\frakp^*$ with left $G_U$-invariant 
1-forms $\Omega^1(D_0)$ on $D_0$. The invariant $k$-forms $\Omega^k(D_0)$ are 
then identified with $\bigwedge^k \frakp^*$.

Write $U=U^+\oplus U^-$, where $U^+=\Span\{u_1,\cdots, u_p\}$ and 
$U^-=Z_0$. Using the basis $v_1,\cdots, v_m$ to 
identify $\frakg \frakl(U)$ with $M_m(\CC)$, an explicit basis for 
$\frakg\subset \frakg \frakl (U)	$ is 
given by
$$\{X_{jk}: j<k\} \cup \{Y_{jk}: j\leq k\},$$
where
\begin{align}
X_{jk} = -E_{jk}+E_{kj},\ \ \ Y_{jk} = -i(E_{jk}+E_{kj}).
\end{align}

Let $\{X_{jk}^*$, $Y_{jk}^*\}$ denote the dual basis for $\frakg^*$. Then a 
basis for $\frakp^*$ is
$$\{ X_{jk}^*,\ Y_{jk}^* : 1\leq j \leq p,\ p+1\leq k \leq m\}.$$

Put
\begin{align} \xi_{jk} = X_{j,k+p}^* + i Y_{j,k+p}^*,\ \ \ 1\leq j \leq p,\ \ \ 
1 \leq k \leq q.\end{align}
Under the identification $\frakp^* \simeq \Omega^1(D)$, $\{\xi_{jk}\}$ is 
a basis for the subspace $(\frakp^*)^+\simeq \Omega^{(1,0)}(D)$ of 
$G_U$-invariant $(1,0)$-forms on $D$. Let $A_{jk}$ 
denote left-multiplication by $\xi_{jk}$ in the exterior algebra $\bigwedge 
\frakp^*$. The unitary 
Howe operator $D^+: \bigwedge \frakp^* \otimes \scrS(V) \rar \bigwedge 
\frakp^{*}
\otimes \scrS(V)$
is defined by
\begin{align}
D^+ = \frac{1}{2^{2q}} \left\{\prod_{k=1}^q \sum_{j=1}^p A_{jk}\otimes 
\left(\ol{z}_j - \frac{1}{\pi} \frac{\partial}{\partial z_j}\right)\right\}.
\end{align}

Write
\begin{align} \varphi_0 = \exp\left(-\pi \sum_{i=1}^m |z_i|^2 \right)\end{align}
for the standard Gaussian on $U$. The Kudla-Millson form of type $(q,q)$ is 
\begin{align}\label{phiqq} \varphi_{q,q}^+=D^+\ol{D^+} \varphi_0\in 
\bigwedge
\frakp^* \otimes \scrS(U).\end{align}
To unwind the expression for $\varphi_{q,q}^+$, we employ
multi-index notation. For a set of integers 
\begin{align}\label{alpha} \alpha = (a_1,\cdots, 
a_q),\ \ \ 1 \leq  a_1  \cdots , a_q \leq 
p\end{align} put
\begin{align} z_\alpha = z_{a_1} \cdots z_{a_q}\end{align}
and
\begin{align} B_\alpha = A_{a_1,1}\wedge A_{a_2,2}\wedge \cdots 
\wedge 
A_{a_q,q},\end{align}
which is a $(q,0)$-form in $\bigwedge \frakp^*$. We have a differential 
operator
\begin{align} \label{Di} D_{i} = z_i - 
\frac{1}{\pi}\frac{\partial}{\partial 
	\ol{z}_i},\end{align}
and its multi-index iteration
\begin{align}\label{multiD}
D_{\alpha} = D_{a_1} D_{a_2} \cdots D_{a_r}.
\end{align}

\begin{lemma} 
	$$\varphi_{q,q}^+ = 2^{-4q} \sum_{\alpha,\alpha'} B_\alpha \wedge 
	\ol{B_{\alpha'}}\otimes D_\alpha \ol{D_{\alpha'}} \varphi_0,$$
	where $\alpha$, $\alpha'$ range over all possible $q$-tuples satisfying 
	(\ref{alpha}).
\end{lemma}
\begin{proof}
	By definition
	$$\varphi_{q,q}^+ = \frac{1}{2^{4q}} \left\{\prod_{k=1}^q \sum_{j=1}^p 
	A_{j,k}\otimes 
	\left(\ol{z}_j - \frac{1}{\pi} \frac{\partial}{\partial 
		z_j}\right)\right\}
	\left\{\prod_{j=1}^p \sum_{k=1}^q \ol{A_{j,k}}\otimes 
	\left(z_j - \frac{1}{\pi} \frac{\partial}{\partial \ol{z}_j}\right)\right\}
	\varphi_0.$$
	Expanding the product, it can be written as
	$$ \frac{1}{2^{4q}} \underset{1 
		\leq 
		a'_1,\cdots ,a'_q \leq p}{\sum_{1\leq 
			a_1,\cdots,a_q\leq p}} A_{a_1,1}\wedge 
	\cdots \wedge A_{a_q,q} \wedge 
	\ol{A_{a_1',1}} \cdots \ol{A_{a_q',q}} 
	\otimes \left(\ol{z_{a_1}} - \frac{1}{\pi}\frac{\partial}{\partial 
		z_{a_1}}\right)\cdots \left( z_{{a}'_q} - 
	\frac{1}{\pi}\frac{\partial}{\partial 
		\ol{z_{a'_q}}}\right) \varphi_0$$
	which is the same as
	$$ \varphi_{q,q}^+ = \frac{1}{2^{4q}} \sum_{\alpha,\alpha'}
	B_\alpha \wedge \ol{B_{\alpha'}} \otimes D_\alpha 
	\ol{D_{\alpha'}} \varphi_0.$$	
\end{proof}

The function $\varphi^+_{q,q}$ lies in $\bigwedge^{{q,q}} \frakp^* 
\otimes \scrS(U)$ on which $U(1,1)\times G_U$ acts, as the described in 
\cite{KM86}. Let $K_U = G_U(\RR)\cap \Ut_{2m}(\CC)$, with the intersection 
taking place in $\GL_{2m}(\CC)\simeq \GL(U\otimes_\RR \CC)$. Corresponding to 
the decomposition $U=U_+ \oplus U_-$ we have $K_U = K_+ 
\times K_-$, 
with $K_{\pm}$ acting on $V_{\pm}$. Then 
$$\frakp^* \simeq (V_{+}\otimes_{E} 
V_-^*)_\CC \cong \Hom_{E}(V_{-},V_{+})\otimes \CC$$
and
$$ \scrS(V)\otimes {\bigwedge}^* \frakp^* \simeq \scrS(V_-) \otimes 
\scrS(V_{+}) \otimes {\bigwedge}^* \Hom_{E}(V_-,V_+)_\CC.$$
The action of $K_+$ is trivial on the Gaussian $\varphi_0$ and commutes with 
the Howe operator $\nabla$, so it leaves $\varphi_{q,q}^+$ invariant. The 
action of $K_-$ is through both factors $\scrS(V_-^n)$ and 
$\Hom_E(V_-,V_+)_\CC$. Of these, the action on $\scrS(V_-^n)$ evidently leaves 
$\varphi_{q,q}^+$ invariant, and the action on $\bigwedge^* 
\Hom_E(V_-,V_+)_\CC$ is by the determinant of $K_- \subset 
\GL({V_{-}\otimes\CC})$ 
\cite[Theorem 
3.1]{KM86}.

For $n>1$, we put
\begin{align}\varphi_{nq,nq}^+ = \varphi_{q,q}^+ \wedge \cdots 
\wedge 
\varphi_{q,q}^+\ \ \ 
(n\text{ times})\end{align}
considered as an element in $\bigwedge^{*} \frakp^* \otimes \scrS(U^n)$ the 
following way. Consider $\varphi_{q,q}^+$ as a Schwartz 
function on $U$ with values in $\bigwedge^{*} \frakp^*$. Then 
$\varphi_{nq,nq}^+$ 
corresponds in the same way to the function
$$U^n \rar {\bigwedge}^* \frakp^*,\ \ \ (u_1, \cdots, u_n)\mapsto 
\varphi_{q,q}^+(u_1)\wedge \cdots \wedge \varphi_{q,q}^+(u_n).$$

We now consider the case $n=p$, so that $\varphi_{pq,pq}^+$ is a sum of 
terms of the form
\begin{align}\label{term} B_{\alpha_{1}}\wedge \ol{B_{{\alpha_{1}}'}} \wedge 
B_{\alpha_2}\wedge 
\ol{B_{\alpha_2'}} \wedge \cdots \wedge B_{\alpha_{p}}\wedge 
\ol{B_{\alpha_p'}}\otimes \prod_{j=1}^p D_{\alpha_j}^{(j)} 
\ol{D_{\alpha_j'}^{(i)}}
\varphi_{0,{(j)}}\end{align}
where $\alpha_1,\cdots, \alpha_p$ are each a $q$-tuple of integers between $1$ 
and $p$, and $D_{\alpha_j}^{(j)}$, $\ol{D_{\alpha'_j}^{(j)}}$ are defined by 
(\ref{multiD}) acting on $\varphi_{0,{(j)}}$, 
the Gaussian on the $j$th copy of $U$ in $U^p$. Since $\bigwedge^{pq,pq} 
\frakp^*$ is one-dimensional, we 
can write
\begin{align} \label{phiKM} \varphi_{pq,pq}^+ = (\omega \wedge \ol{\omega}) 
\otimes 
\varphi_{\KM}\end{align}
for some $\varphi_{\KM}\in \scrS(U^p)$, with
\begin{align}\label{w} \omega = \xi_{1,1}\wedge\xi_{2,1}\wedge \cdots \wedge 
\xi_{p,1} 
\wedge 
\xi_{1,2}\wedge 
\cdots \wedge \xi_{p,2} \wedge \cdots \wedge \xi_{1,q}\wedge \cdots 
\xi_{p,q}.\end{align}

Let $\{z_i^{(j)}: 1\leq i \leq m,\ 1 \leq j \leq p\}$ denote the coordinates on 
$U^p= \prod_{j=1}^p U^{(j)}$, where for each $j$, $\{z_1^{(j)},\cdots, 
z_{m}^{(j)}\}$ are the standard 
coordinates for $U^{(j)}$.  Then
$$ \Phi_0 = \exp\left(-\pi \sum_{i=1}^m\sum_{j=1}^p |z_i^{(j)}|^2\right)$$
is the Gaussian on $U^p$.
\begin{lemma}\label{phiKMlem} The following formula holds:
	\begin{align} \varphi_{\KM} = (-1)^{pq(p-1)/2} 
	\sum_{\sigma, \sigma'\in S_p^q} D_{\sigma, \sigma'} \Phi_0 \in 
	\scrS(U^p),\end{align}
	where $S_p$ denotes the symmetric group on $\{1,\cdots, p\}$ and 
	$\sigma=(\sigma_1,\cdots, \sigma_q)$, $\sigma$ range over all pairs 
	of $q$-tuples of elements of the symmetric group $S_p$, and
	$$ D_{{\sigma}, {\sigma'}} = \prod_{j=1}^p \prod_{k=1}^q 
	D_{\sigma_k(j)}^{(j)} 
	\ol{D_{\sigma'_k(j)}^{(j)}}.$$
\end{lemma}
\begin{proof} In the expression (\ref{term}), writing $\alpha_j = 
	(a_{j1},\cdots, a_{jq})$ and $\alpha_j'=(a_{j1}',\cdots, a_{jq}')$ for each 
	$j$, the wedge product factor on the left is
	\begin{align}\label{xis} (\xi_{a_{11},1} \wedge \cdots 
	\xi_{a_{1q},q})\wedge 
	(\ol{\xi_{a'_{11},1}}\wedge
	\cdots \ol{\xi_{a'_{1q},q}})\wedge \cdots \wedge (\xi_{a_{p1},1} \wedge 
	\cdots \xi_{a_{pq},q})\wedge (\ol{\xi_{a'_{p1},1}} \wedge \cdots 
	\ol{\xi_{a'_{pq},q}}).
	\end{align}
	
	For this expression not to vanish there must be no repeated wedge factors, 
	so 
	that it coincides with $\omega \wedge \ol{\omega}$ up to sign. This amounts 
	to 
	the condition
	$$\{a_{1k},\cdots,a_{pk}\} = \{ 
	a_{1k}',\cdots, a_{pk}'\} = \{1,\cdots, p\},\ \ 1 \leq k \leq q.$$
	In other words, $\sigma_k(j) = a_{jk}$ and $\sigma_k'(j)=a_{jk}'$ define 
	elements $\sigma_k, \sigma'_k$ in $S_p$. Then writing 
	$\xi_{\sigma,\sigma'}$ 
	for (\ref{xis}), we get 
	$$ \varphi_{pq,pq}^+ =\sum_{\sigma,\sigma'\in S_p^k} \xi_{\sigma,\sigma'} 
	\otimes D_{\sigma,\sigma'} \Phi_0.$$
	
	It's enough to prove that $\xi_{\sigma,\sigma'}=(-1)^{pq(p-1)/2}\omega 
	\wedge \omega'$.  To show this we perform a sequence of transpositions that 
	transforms (\ref{xis}) to $\omega \wedge \ol{\omega}$ and count the total 
	sign change. We do this in two stages.
	
	In the first stage, we sort the terms $\xi_{i,j}$. We find the 
	\textit{last} term in the desired order among them, i.e. $\xi_{p,q}$, and 
	shift it 
	left until it is in the leftmost spot, where $\xi_{1,1}$ is in (\ref{w}). 
	Then 
	we take the second last term
	$\xi_{p-1,q}$ and also shift it all the way to the left, so that now 
	$\xi_{p,q}$ 
	becomes the second term from the left. We continue with $\xi_{p-2,q}$, etc. 
	shifting the terms in strictly reverse order to the leftmost spot until at 
	the 
	end we move $\xi_{1,1}$ to the leftmost spot. At that point all the 
	$\xi_{i,j}$ terms will be in the same order as in (\ref{w}), while the 
	$\ol{\xi_{i,j}}$ terms remain unsorted. 
	
	In the second stage, we sort the $\ol{\xi_{i,j}}$ terms. We take the term 
	$\ol{\xi_{p,q}}$ and shift it to the left until it is in position $pq+1$ 
	from 
	the left, to the right of $\xi_{p,q}$. 
	Then 
	we take the second last term $\ol{\xi_{p,q-1}}$ and shift it also to 
	position 
	$pq+1$, between $\xi_{p,q}$ and $\ol{\xi_{p,q}}$. We continue again 
	in reverse order shifting all terms $\ol{\xi_{i,j}}$ to position $pq+1$ 
	until 
	at 
	the end we shift $\ol{\xi_{1,1}}$ into that spot. At this point all terms 
	are 
	in the correct order and the expression is $\omega 
	\wedge \ol{\omega}$ up to sign.
	
	Now we count how many sign changes were made in each stage. In the first 
	stage, 
	when moving a term $\xi$ we distinguish two kinds of 
	transpositions of adjacent terms: if $\xi$ is transposed with one of the 
	terms 
	$\ol{\xi_{i,j}}$, we consider it of the first kind. If it is 
	transposed with one of $\xi_{i,j}$, we consider it of the second kind. 
	
	We count the transpositions of the first kind first. Note that each term 
	$\xi_{i,j}$ 
	in $B_{\alpha_p}$ has to pass every term $\ol{\xi_{i,j}}$ occurring in 
	$B_{\alpha_1'}, 
	B_{\alpha_2'}, 
	\cdots, B_{\alpha_{p-1}'}$, therefore it makes $q(p-1)$ passes of the 
	first kind. Altogether the $q$ terms of $B_{\alpha_p}$ make $q^2(p-1)$ 
	passes of the first kind. The terms $\xi_{i,j}$ in $B_{\alpha_{p-1}}$ have 
	to pass every $\ol{\xi_{i,j}}$ term in $B_{\alpha_1'}, B_{\alpha_2'}, 
	\cdots, B_{\alpha_{p-2}'}$ exactly once, so altogether they make $q^2(p-2)$ 
	passes of the first kind. Terms in $B_{\alpha_{p-2}}$ make
	$q^2(p-3)$ passes of the first kind, etc., and altogether the number of 
	passes 
	of the first kind in the first stage is $q^2(p-1) + q^2(p-2) + \cdots + 
	q^2=q^2(p-1)p/2$, which is $pq(p-1)/2$ modulo $2$.
	
	Now we observe that the number of transpositions of the second kind in the 
	first stage is equal to the total number of transpositions in the second 
	stage. 
	Thus their sign contributions cancel out, and the total sign change is 
	$(-1)^{\frac{pq(p-1)}{2}}$ as claimed.
\end{proof}

Recall the hermitian space $(V,Q)=(V_0,Q_0)\otimes_{E_0} E$ from the 
introduction and put $U=V(\RR)$. Let $p_i: D=D_0^g \rar 
D_0$ denote the $i$th projection and put 
\begin{align*} \varphi_{\infty} = p_1^*(\varphi_{nq,nq}^+) \wedge \cdots \wedge 
p_g^*(\varphi_{nq,nq}^+)\in  \Omega^{gnq,gnq}(D)\otimes 
\scrS(V_0(\RR)^{ng}),\end{align*}
where identify $\scrS(V(\RR)^n)$ and $\scrS(V_0(\RR)^{ng})$ using the 
isomorphism from (\ref{V0^nV}). If 
$\Delta: D_0 
\rar D$ is the diagonal map, then as explained in \cite{KM86}
\begin{align} \Delta^* \varphi_\infty = \varphi_{ngq,ngq}^+.
\end{align}

In particular, if $ng=p$, then 
\begin{align} \Delta^* \varphi_\infty = \omega \wedge \ol{\omega} 
\otimes \varphi_{\KM} \in \Omega^{pq,pq}(D_0)\otimes 
\scrS(V_0(\RR)^{ng}).\end{align}

\subsection{Vanishing of the Ikeda map}

For $z$ one of the variables $z_1,\cdots, z_m$ and non-negative integers $a$ 
and $b$ we define the function $F_{a,b}(z)$ as a polynomial in $z$ and $\ol{z}$ 
by the relation
\begin{align} \label{Fab} \left(\ol{z} - \frac{1}{\pi} \frac{\partial}{\partial 
	z}\right)^a\left(z - 
\frac{1}{\pi} \frac{\partial}{\partial \ol{z}}\right)^b \varphi_0 = 
F_{a,b}(z) 
\cdot \varphi_0.\end{align}
For a rapidly decreasing function $f(z)$ on $\CC$, put \begin{align}I(f) = 
\int_{\CC} f(z) e^{-2\pi |z|^2} dz.
\end{align}
For non-negative integers $a$ and $b$, $a\neq b$ it's easy to verify that
\begin{align}\label{zazb}I(P_{a,b})=0,\ \ \ P_{a,b}(z) = \ol{z}^a 
z^b\end{align}
by sign considerations. 
\begin{lemma}\label{lemFab}
	$I(F_{a,b})=0$ if $a\neq b$.
\end{lemma}
\begin{proof}
	Assume $b>a$ without loss of generality. First, a simple induction argument shows
	$$ F_{0,b} = (2 z)^b.$$
	Then $F_{a,b}$ is determined by
	$$ \left(\ol{z} - \frac{1}{\pi} \frac{\partial}{\partial 
		z}\right)^a (2z)^b \varphi_0 = F_{a,b}\varphi_0.$$
	Note that for $\alpha\geq 0, \beta\geq 1$, we have
	$$ \left(\ol{z} - \frac{1}{\pi} \frac{\partial}{\partial 
		z}\right) \ol{z}^\alpha z^\beta \varphi_0 = (2\ol{z}^{\alpha+1} 
	z^\beta-\frac{\beta}{\pi} 
	\ol{z}^{\alpha} z^{\beta-1})\varphi_0.$$
	
	Now let $f(z)$ be a polynomial in $z$ and $\ol{z}$ such that the 
	coefficient of 
	$\ol{z}^{\alpha} z^\beta$ is non-zero only if $\beta>\alpha$. Let 
	$\mu(f)$ be the 
	minimum of $\beta-\alpha$ among all terms $\ol{z}^{\alpha} z^\beta$ of 
	$f$ with non-zero coefficients. The above 
	identity shows that if $g(z)$ is defined by
	$$ g(z) \varphi_0 = \left(\ol{z} - \frac{1}{\pi} \frac{\partial}{\partial 
		z}\right) f(z) \varphi_0,$$
	then 
	$$ \mu(g) \geq \mu(f)-1,\ \ \ \text{if}\ g\not\equiv 0$$
	Since $\mu(F_{0,b})=b$, it follows that if $F_{a,b}\not\equiv 0$, then 
	$\mu(F_{a,b}) 
	\geq b-a>0$. Then every 
	non-zero term of $F_{a,b}$ is of the form $\ol{z}^\alpha z^\beta$, with 
	$\beta>\alpha$. The lemma then follows from (\ref{zazb}).
\end{proof}

Next we set out to compute $I(f_k(z))$, for $f_k(z)=F_{k,k}(z)$, $k\geq 0$. The 
functions $f_k$ are determined by 
\begin{align} E^k \varphi_0 = f_k(z) \varphi_0,\ \ \ f_0(z)=1,\ \ \ k\geq 
0\end{align}
where
\begin{align} E = \left(\ol{z} - \frac{1}{\pi}\frac{\partial}{\partial 
	z}\right)\left(z - 
\frac{1}{\pi}\frac{\partial}{\partial \ol{z}}\right).\end{align}
For any smooth function $f(z,\ol{z})$, we have
\begin{align*} E (f \varphi_0) = \left\{ \left( 4|z|^2 - \frac{2}{\pi}\right)f 
- 
\frac{2}{\pi}\left(\ol{z}\frac{\partial f}{\partial \ol{z}} + z \frac{\partial 
	f}{\partial z}\right) + \frac{1}{\pi^2} \frac{\partial^2 f}{\partial z 
	\partial \ol{z}} \right\}  \varphi_0.\end{align*}

We normalize by changing variables $z = \frac{w}{\sqrt{2\pi}}$ and setting
\begin{align}\label{gk}
g_k(w) = f_k\left(\frac{w}{\sqrt{2\pi}}\right) \frac{\pi^k}{2^k}
\end{align}
to obtain the recursive relation
\begin{align}g_{k+1}(w)= \left\{ \left( |w|^2 - 1 \right) g_k - 
\left( \ol{w} 
\frac{\partial g_k}{\partial \ol{w}} + w \frac{\partial g_k}{\partial w}\right)
+ 
\frac{\partial^2 g_k}{\partial w \partial \ol{w}}\right\},\ \ \ 
g_0(w)=1
\end{align}

Then $g_k(w)$ is the $k$th Laguerre polynomial in $|w|^2$, normalized to be 
monic and have integer coefficients, e.g. $g_1(w) = |w|^2 - 1$.
\begin{prop}\label{gkexp} For $k\geq 1$,
	$$g_k(w) = \sum_{r=0}^k (-1)^{r+k} |w|^{2r} \frac{(k!)^2}{(r!)^2(k-r)!}.$$
\end{prop}
\begin{proof}
	The formula is evidently valid for $k=1$. Assume it holds for some $k$. 
	Noting 
	that for $r>0$,
	$$ \ol{w} \frac{\partial}{\partial \ol{w}} |w|^{2r} = 
	w \frac{\partial}{\partial w}|w|^{2r} = r |w|^{2r},\ \ \ 
	\frac{\partial^2}{\partial w \partial \ol{w}} |w|^{2r} = r^2 |w|^{2r-2}$$
	we have 
	$$ g_{k+1}(w) = (-1)^k (|w|^2-1)(k!) + \sum_{r=1}^k (-1)^{r+k} 
	\frac{(k!)^2}{(r!)^2 
		(k-r)!} \left\{ |w|^{2r+2} - |w|^{2r} - 2r |w|^{2r} + 
	r^2|w|^{2r-2}\right\}$$
	
	For $0<r<k$, the coefficient of $|w|^{2r}$ in the sum is equal to
	\begin{align*} &(-1)^{k+r+1} k! \left\{ \frac{k!}{((r-1)!)^2 (k-r+1)!} + 
	(1+2r) 
	\frac{k!}{(r!)^2(k-r)!} + (r+1)^2 
	\frac{k!}{((r+1)!)^2(k-r-1)!}\right\}\\
	&=(-1)^{k+r+1} \frac{(k!)^2}{(r)!^2(k-r+1)!}\left(r^2 + (1+2r)(k-r+1) + 
	(k-r)(k-r+1)\right)\\
	&=(-1)^{k+r+1}\frac{(k+1)!^2}{r!^2(k-r+1)!}
	\end{align*}
	
	The leading coefficient of $g_{k+1}(w)$  is 
	$\frac{(k!)^2}{(k!)^2(k-k)!}=1$, and the 
	constant term is 
	$$(-1)^{k+1}(k!) + (-1)^{k+1}\frac{(k!)^2}{(1!)^2(k-1)!}\cdot 1^2 = 
	(-1)^{k+1}(k+1)!.$$
	All coefficients match those in the formula, which is therefore valid for 
	$k+1$, and by induction for all $k$.
\end{proof}

Consider now a two-dimensional complex hermitian space of signature $(1,1)$, 
with standard coordinates $(x_1,x_2)$. Put
$$ F_k(x_1,x_2) = f_k(x_1) e^{-\pi(|x_1|^2+|x_2|^2)}.$$

It can be split into a sum of two isotropic lines via the change of 
variables $(w_1,w_2)=(\frac{x_1+x_2}{\sqrt{2}},\frac{x_1-x_2}{\sqrt{2}}).$
In these coordinates $F_k$ is given by
\begin{align}\label{Fx1x2} 
F_k(x_1,x_2)=f_k\left(\frac{w_1+w_2}{\sqrt{2}}\right)e^{-\pi(|w_1|^2+|w_2|^2)}.
\end{align}

The computation of the integral of $F_k$ along the isotropic line $(w_1,0)$ is 
the key lemma of this section.

\begin{lemma}\label{fk=0} $$\int_{\CC} 
	f_k\left(\frac{z}{\sqrt{2}}\right)e^{-\pi |z|^2} dz
	=0 $$
\end{lemma}
\begin{proof}Using the relation $f_k(z) = g_k(\sqrt{2\pi} z)\frac{2^k}{\pi^k}$ 
	and the formula for $g_k$ from Proposition \ref{gkexp}, we express the 
	integral as
	$$\frac{2^k(-1)^k \cdot (k!)}{\pi^k} \sum_{r=0}^k 
	\binom{k}{r}\frac{(-1)^r\pi^r}{r^!} \int_{\CC} |z|^{2r} 
	e^{-\pi|z|^2}  
	|z|^{2r} dz.$$
	
	Since
	$$ 2\pi \int_{0}^\infty x^{2r+1} e^{-\pi x^2} dx = \frac{r!}{\pi^{r}}$$
	the integral is
	$$ \frac{2^k (-1)^k k!}{\pi^k} \sum_{r=0}^k \binom{k}{r} (-1)^r = 0.$$
\end{proof}

We now assume $ng=p$ and $q<p$. Let 
$\varphi_L$ denote 
the characteristic function of $\wh{L}\subset V(\Ads_f)$, where $L\subset V$ is a maximal lattice. Put
\begin{align}
\varphi = \varphi_{\KM}\otimes \varphi_L \in \scrS(V(\Ads)^n).
\end{align}

Since $q<p$, the Weil index $r$ of $(V_0,Q_0)$ is $q$. Then the numbers $r$, 
$n_0 = \dim W_0 = ng$, and $m=p+q = \dim V_0$, satisfy
$$ n_0 < m \leq 2n_0,\ \ \ m-r \leq n_0.$$

These are parameters for the pair $(W_0,J_0)$ and $(V_0,Q_0)$ falling 
outside the classical convergence range for the Siegel-Weil formula 
\cite{Weil65Acta}, and inside the range for the regularized formula of Ichino 
\cite{Ich04MZ}.

Let $r_0=m-n_0 = q$. Suppose that $\{u_1,\cdots, u_m\}$ is the fixed standard 
basis of $V_0(\RR)$ used to define $\varphi_0$. for $j=1,\cdots, r_0$, put
$$ e_j = \frac{u_j + u_{m+1-j}}{\sqrt{2}},\ \ \ f_j = \frac{u_j - 
	u_{m+1-j}}{\sqrt{2}}.$$
Then $\{e_i, f_j: 1 \leq i,j \leq r_0\}$ form a basis for a subspace $V''_0$ of 
$V_0(\RR)$ such that
$$ Q_0(e_i, f_j) = \delta_{i,j},\ \ \ Q_0(f_i,f_j)=Q_0(e_i,e_j)=0.$$
Put  $Q''_0=Q_0|_{V''_0},$ 
and 
\begin{align}(V'_0,Q'_0) = 
(V''_0,Q''_0)^\perp,\end{align}
so that $(V_0,Q_0) = (V'_0,Q'_0) + (V''_0,Q''_0)$ is an orthogonal 
decomposition.

Let $\beta'=\{v'_1,\cdots, v'_{p-q}\}$ be any basis for $(V'_0,Q'_0)$
and put $\beta=\{e_1,\cdots, e_{r_0},v'_1,\cdots, 
v'_{p-q},f_1,\cdots f_{r_0}\}$. Then 
\begin{align*} [Q_0]_{\beta} =  \left(\begin{array}{ccc}& & I_{p} \\ & 
[Q'_0]_{\beta'}& \\
I_{p}& &\end{array}\right).
\end{align*}

Using this matrix presentation we take $P_0$ to be following parabolic subgroup of $H_0$:
\begin{align} P_0(R) = \left\{ \left(\begin{array}{ccc} A  & * & * \\ & X & 
* \\ & & 
{}^t \ol{A}^{-1}\end{array}\right): A\in \GL_{r_0}(R\otimes E_0),\ X\in 
H'_0(R)\right\},\ \ \ R\in \QQ\dash \Alg
\end{align}
where $H_0'=U(V_0',Q_0')\simeq U(p-q;E_0/\QQ)$.

We choose a maximal compact subgroup $K_0=\prod_v K_{v}$ of $H(\Ads)$ as 
follows. For $v$ a rational prime, $K_v$ is the stabilizer of 
$L_v \subset V_0(\QQ_v)$. For $v=\infty$, we take $K_\infty = H_0(\RR) \cap 
U_{2m}(\CC)$, with $U_{2m}(\CC)$ and $H_0(\RR)$ considered as subgroups of 
$\GL_{2m}(\CC)\simeq \GL(V_0(\RR))$.

The regularization procedure in \cite{Ich04MZ} involves applying to a function
$\varphi_0\in \scrS(V_0(\RR)^{ng})$
the ($g$-fold tensor power of) maps
$$ \pi : \scrS(V_0(\Ads)^{n}) \rar  \scrS(V_0'(\Ads)^{n}),\ \ \ \pi = 
\Ik \circ \pi_{K_0} ,$$
where 
\begin{align} \pi_{K_0}: \scrS(V_0(\Ads)^n) \rar \scrS(V_0(\Ads)^n),\ \ \ 
\label{piK} \pi_{K_0}(\psi)(x) 
= \int_{K_0} \psi(kx) 
dk\end{align}
and $\Ik$ is the Ikeda map
\begin{align}
\Ik: \scrS(V_0(\Ads)^n) \rar \scrS (V_0'(\Ads)^n),\ \ \ \Ik(\psi)(x) = 
\pi_{Q'}^{Q}(\psi)(x) = \int_{V''_e(\Ads)^n} 
\Psi\left(\begin{array}{c}y \\ x 
\\ 0 \end{array}\right) dy 
\end{align}
with
\begin{align} V_e'' = \Span_{E}\{e_1,\cdots, e_{r_0}\} \subset V''_0,\ \ \ \ 
V_{f}'' = \Span_{E}\{f_1,\cdots,f_{r_0}\},\ \ \ V_0 = V_e'' \oplus V'_0 \oplus 
V_{f}''.\end{align}

\begin{thm}\label{Ikedakills} $\varphi_{KM}$ lies in the kernel of the Ikeda map $\scrS(V_0(\Ads)^{p}) \rar \scrS(V_0'(\Ads)^{p})$.
\end{thm}

\begin{proof} 
	First we note that $\varphi_{KM}$ is $K_0(\RR)$-invariant, i.e. 
	$$(k_\infty x)=\varphi_{KM}(x),\ \ \ k_\infty 
	\in K_0(\RR).$$ This follows from the description of the action of 
	$K_0(\RR)\subset H_0(\RR)$ 
	on $\scrS(V_0(\RR))$ given in the previous section. The essential facts are 
	that the Gaussian $\varphi_0$ is $K_0(\RR)$-invariant, and the Howe operators 
	$D_\alpha$ commute with the action of $K_0(\RR)$.
	
	Since $\varphi$ is factorizable, the adelic integral in $\Ik(\varphi)$ vanishes 
	if any of the corresponding local integral factors do. Then it suffices to show 
	$$ \int_{V_e''(\RR)^{p}} \varphi_{\KM}\left(\begin{array}{c}y \\x \\ 0 
	\end{array}\right)dy=0.$$
	
	By Lemma (\ref{phiKMlem}), $\varphi_\KM$ is up to sign a sum of terms of the 
	form
	$$ \prod_{j=1}^p \left( \prod_{k=1}^q D_{\sigma_k(j)}^{(j)} 
	D_{\sigma_k'(j)}^{(j)}\varphi_{0}^{(j)}\right).$$
	
	We claim that for such terms, for each $j$,
	$$ \int_{V_e''(\RR)} \left( \prod_{k=1}^q D_{\sigma_k(j)}
	D_{\sigma_k'(j)}\varphi_{0}\right)\left(\begin{array}{c} y \\ 
	x \\ 0 \end{array}\right)= 0$$
	from which the theorem follows. Writing $a_k = \sigma_{k}(j)$, $a'_k = 
	\sigma'_{k}(j)$, the integrand as a function of standard coordinates 
	$(z_1,\cdots, z_m)$ on $V_0(\RR)$ is
	$$\prod_{k=1}^q \left(\ol{z_{a_k}} - \frac{1}{\pi}\frac{\partial}{\partial 
		z_{a_k}}\right)\left(z_{a'_k} - 
	\frac{1}{\pi}\frac{\partial}{\partial \ol{z_{a'_k}}}\right) \varphi_0.$$
	The domain $V_e''(\RR)$ is a direct sum of $q$ isotropic 
	lines spanned by $e_k$, for $k=1,\cdots, q$. Fix some $k$, and let $a$, 
	resp. $b$ denote the multiplicity of $k$ in $(a_1,\cdots, a_q)$, resp. $(a_1
	,\cdots, a_q')$. We show the factor of the integral corresponding to the 
	isotropic line $\Span\{e_k\}$ vanishes by distinguishing two cases:
	
	Case 1: $a\neq b$. This integral was considered in Lemma \ref{lemFab}. It vanishes by argument considerations.
	
	Case 2: $a=b$. This integral was explicitly computed and shown to vanish in Lemma \ref{fk=0}.
\end{proof}

\bibliographystyle{alpha}
{\small \bibliography{refdb}}

\begin{thebibliography}{GQT14}

\bibitem[FM02]{FunkMill02}
Jens Funke and John Millson.
\newblock Cycles in hyperbolic manifolds of non-compact type and {F}ourier
  coefficients of {S}iegel modular forms.
\newblock {\em Manuscripta Math.}, 107(4):409--444, 2002.

\bibitem[FM06]{FunkMill06}
Jens Funke and John Millson.
\newblock Cycles with local coefficients for orthogonal groups and
  vector-valued {S}iegel modular forms.
\newblock {\em Amer. J. Math.}, 128(4):899--948, 2006.

\bibitem[FM11]{FunkMill11}
Jens Funke and John Millson.
\newblock Spectacle cycles with coefficients and modular forms of half-integral
  weight.
\newblock In {\em Arithmetic geometry and automorphic forms}, volume~19 of {\em
  Adv. Lect. Math. (ALM)}, pages 91--154. Int. Press, Somerville, MA, 2011.

\bibitem[FM13]{FunkMill13}
Jens Funke and John Millson.
\newblock Boundary behaviour of special cohomology classes arising from the
  {W}eil representation.
\newblock {\em J. Inst. Math. Jussieu}, 12(3):571--634, 2013.

\bibitem[GI14]{GanIchino}
Wee~Teck Gan and Atsushi Ichino.
\newblock Formal degrees and local theta correspondence.
\newblock {\em Invent. Math.}, 195(3):509--672, 2014.

\bibitem[GQT14]{GQT}
Wee~Teck Gan, Yannan Qiu, and Shuichiro Takeda.
\newblock The regularized {S}iegel-{W}eil formula (the second term identity)
  and the {R}allis inner product formula.
\newblock {\em Invent. Math.}, 198(3):739--831, 2014.

\bibitem[Ich04]{Ich04MZ}
Atsushi Ichino.
\newblock A regularized {S}iegel-{W}eil formula for unitary groups.
\newblock {\em Math. Z.}, 247(2):241--277, 2004.

\bibitem[K87]{KM87}
Stephen~S. Kudla and John~J. ~.
\newblock The theta correspondence and harmonic forms. {II}.
\newblock {\em Math. Ann.}, 277(2):267--314, 1987.

\bibitem[KM86]{KM86}
Stephen~S. Kudla and John~J. Millson.
\newblock The theta correspondence and harmonic forms. {I}.
\newblock {\em Math. Ann.}, 274(3):353--378, 1986.

\bibitem[KM88]{KM88}
Stephen~S. Kudla and John~J. Millson.
\newblock Tubes, cohomology with growth conditions and an application to the
  theta correspondence.
\newblock {\em Canad. J. Math.}, 40(1):1--37, 1988.

\bibitem[KM90]{KM90IHES}
Stephen~S. Kudla and John~J. Millson.
\newblock Intersection numbers of cycles on locally symmetric spaces and
  {F}ourier coefficients of holomorphic modular forms in several complex
  variables.
\newblock {\em Inst. Hautes \'Etudes Sci. Publ. Math.}, (71):121--172, 1990.

\bibitem[KR88]{KRa88}
Stephen~S. Kudla and Stephen Rallis.
\newblock On the {W}eil-{S}iegel formula.
\newblock {\em J. Reine Angew. Math.}, 387:1--68, 1988.

\bibitem[Kud78]{KBalls78}
Stephen~S. Kudla.
\newblock Intersection numbers for quotients of the complex {$2$}-ball and
  {H}ilbert modular forms.
\newblock {\em Invent. Math.}, 47(2):189--208, 1978.

\bibitem[Kud84]{K83}
Stephen~S. Kudla.
\newblock Seesaw dual reductive pairs.
\newblock In {\em Automorphic forms of several variables ({K}atata, 1983)},
  volume~46 of {\em Progr. Math.}, pages 244--268. Birkh\"auser Boston, Boston,
  MA, 1984.

\bibitem[LR05]{LapRal}
Erez~M. Lapid and Stephen Rallis.
\newblock On the local factors of representations of classical groups.
\newblock In {\em Automorphic representations, {$L$}-functions and
  applications: progress and prospects}, volume~11 of {\em Ohio State Univ.
  Math. Res. Inst. Publ.}, pages 309--359. de Gruyter, Berlin, 2005.

\bibitem[Shi64]{Shim64}
Goro Shimura.
\newblock Arithmetic of unitary groups.
\newblock {\em Ann. of Math. (2)}, 79:369--409, 1964.

\bibitem[Wei65]{Weil65Acta}
Andr{\'e} Weil.
\newblock Sur la formule de {S}iegel dans la th\'eorie des groupes classiques.
\newblock {\em Acta Math.}, 113:1--87, 1965.

\end{thebibliography}

\end{document}